%% file: AnnularPaper.tex
\documentclass[10pt, reqno]{amsart}
\pdfoutput=1
\theoremstyle{plain}
\usepackage{amsmath, amsfonts, amsthm, amssymb,amscd,bbold}
\usepackage{graphics}
\usepackage{color}
\usepackage{epsfig}
\usepackage{tikz}
\usepackage{tikz-cd}

\usepackage{bbm}
\usepackage{enumitem}
\usepackage{chngcntr}
\usepackage{float}
\usepackage[export]{adjustbox}
\usepackage{stackengine}
\usepackage{array}
\usepackage{caption}
\usepackage[normalem]{ulem}
\usepackage{tcolorbox}
\usepackage{ stmaryrd }

\usepackage{calc}
\usepackage{tikz}

\usepackage{url}

\usetikzlibrary{arrows}
\usepackage[all]{xy}
\usepackage{todonotes}
\theoremstyle{plain}
\newtheorem{theorem}{Theorem}
\newtheorem{lemma}[theorem]{Lemma}

\newtheorem{question}[theorem]{Question}

\theoremstyle{definition}

\newtheorem{definition}[theorem]{Definition}

\newtheorem*{remark}{Remark}
\newtheorem{example}[theorem]{Example}

\newcommand{\ZZ}{\ensuremath{\mathbb{Z}}}
\newcommand{\XX}{\ensuremath{\mathbb{X}}}

\newcommand{\FF}{\ensuremath{\mathbb{F}}}

\DeclareMathOperator{\Id}{Id}

\newcommand{\OO}{\ensuremath{\mathbb{O}}}

\newcommand{\C}{\ensuremath{\mathcal{C}}}
\newcommand{\D}{\mathcal{D}}
\newcommand{\g}{\mathfrak{g}}

\DeclareMathOperator{\AKh}{AKh}

\newcommand{\CKh}{\ensuremath{\mbox{CKh}}}

\newcommand{\sltwo}{\ensuremath{\mathfrak{sl}_2}}
\newcommand{\slk}[1]{\ensuremath{\mathfrak{sl}_{#1}}}

\newcommand{\exsltwo}{\sltwo(\wedge)}

\newcommand{\RIa}{\llbracket\includegraphics[height=.1in, width=0.15in]{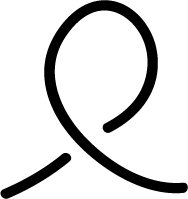}\rrbracket}

\newcommand{\RIb}{\llbracket\includegraphics[height=.1in, width=0.15in]{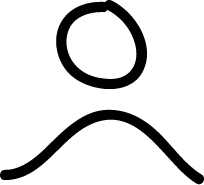}\rrbracket}
\newcommand{\RIc}{\llbracket\includegraphics[height=.1in, width=0.15in]{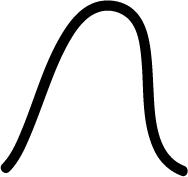}\rrbracket}

\newcommand{\RIIa}{\llbracket\includegraphics[height=.1in, width=0.25in]{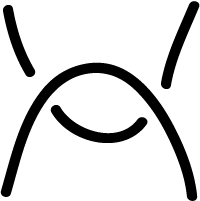}\rrbracket}
\newcommand{\RIIb}{\llbracket\includegraphics[width=0.25in]{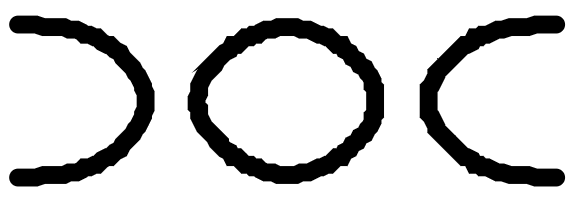}\rrbracket}
\newcommand{\RIIc}{\llbracket\includegraphics[width=0.25in]{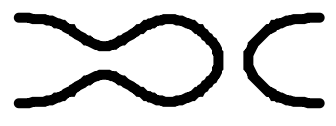}\rrbracket}
\newcommand{\RIId}{\llbracket\includegraphics[width=0.25in]{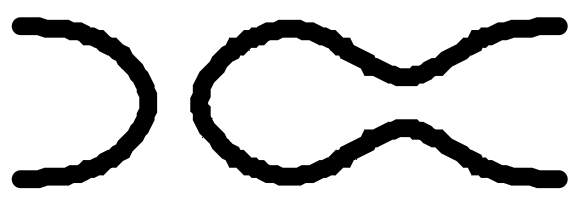}\rrbracket}
\newcommand{\RIIe}{\llbracket\includegraphics[width=0.25in]{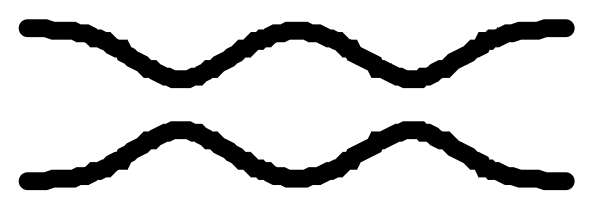}\rrbracket}

\newcommand{\RIIILooo}{\includegraphics[height=0.25in]{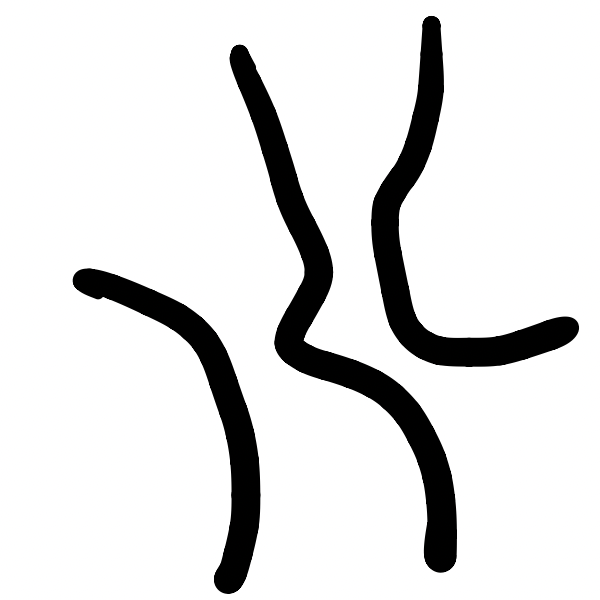}}
\newcommand{\RIIILooi}{\includegraphics[height=0.25in]{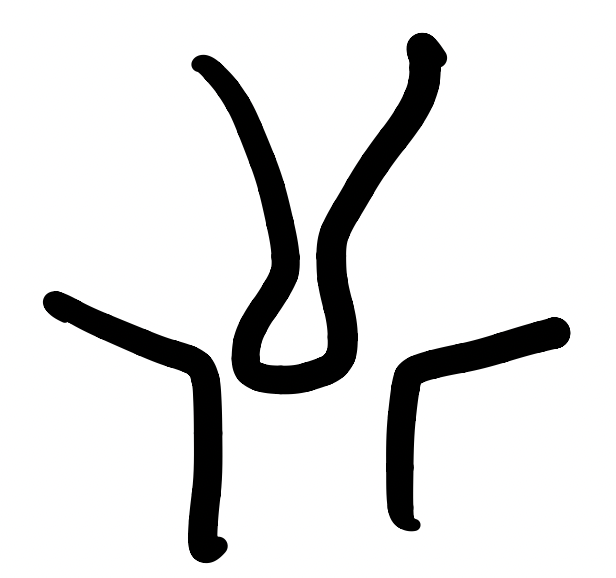}}
\newcommand{\RIIILoio}{\includegraphics[height=0.25in]{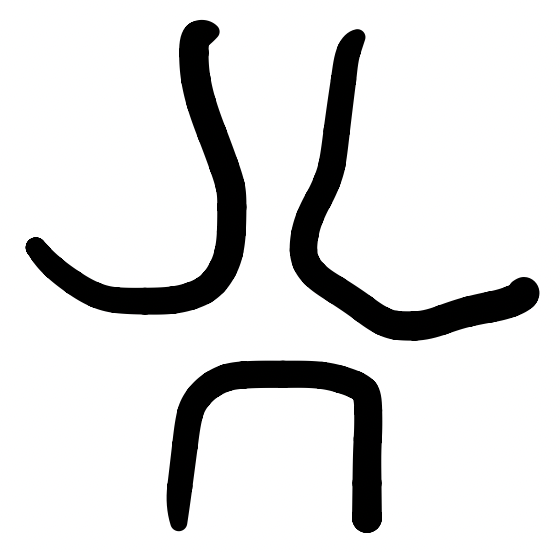}}
\newcommand{\RIIILoii}{\includegraphics[height=0.25in]{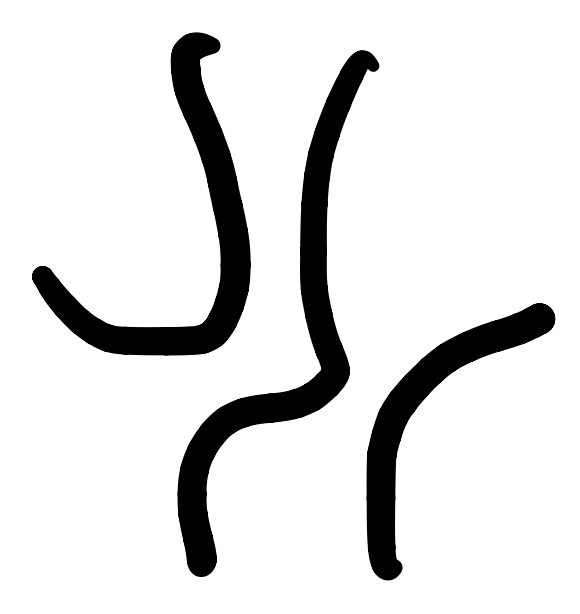}}
\newcommand{\RIIILioo}{\includegraphics[height=0.25in]{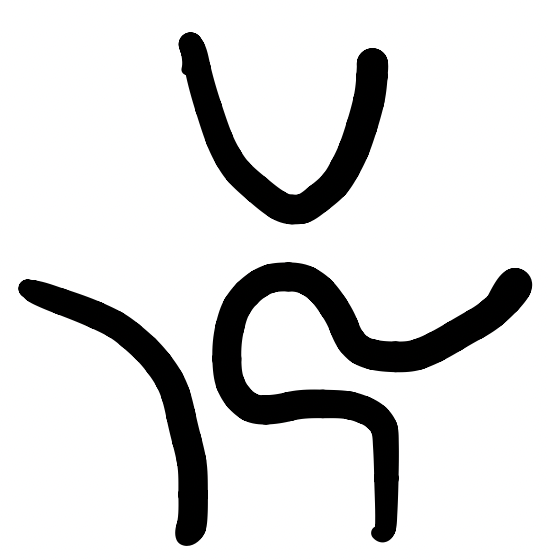}}
\newcommand{\RIIILioi}{\includegraphics[height=0.25in]{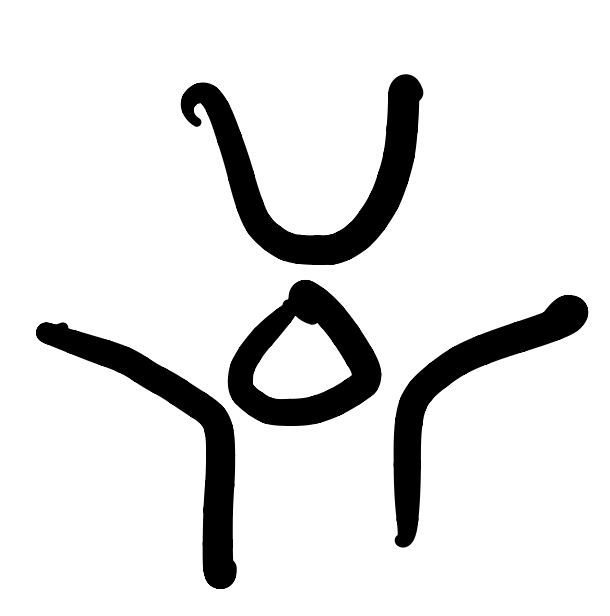}}
\newcommand{\RIIILiio}{\includegraphics[height=0.25in]{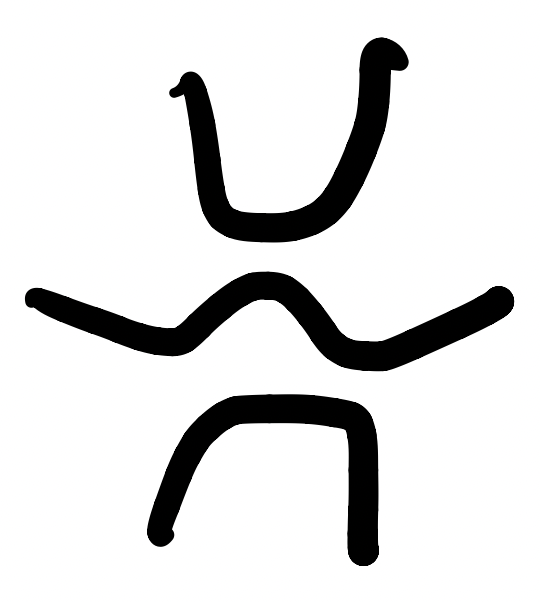}}
\newcommand{\RIIILiii}{\includegraphics[height=0.25in]{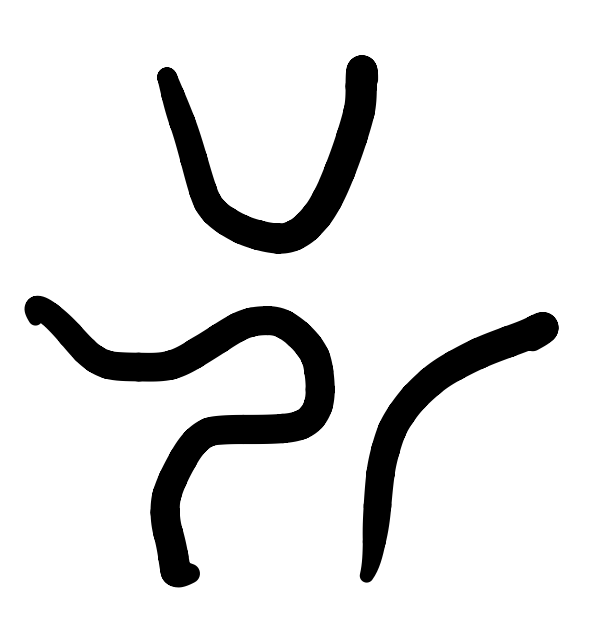}}

\newcommand{\RIIIRooo}{\includegraphics[height=0.25in]{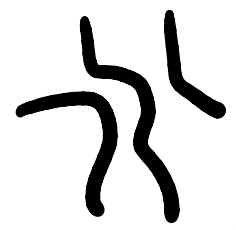}}
\newcommand{\RIIIRooi}{\includegraphics[height=0.25in]{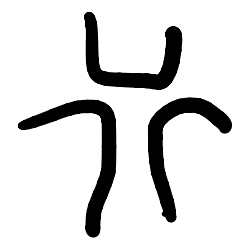}}
\newcommand{\RIIIRoio}{\includegraphics[height=0.25in]{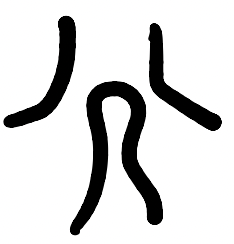}}
\newcommand{\RIIIRoii}{\includegraphics[height=0.25in]{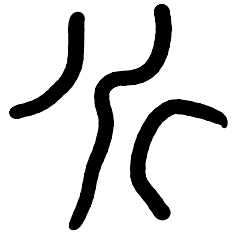}}
\newcommand{\RIIIRioo}{\includegraphics[height=0.25in]{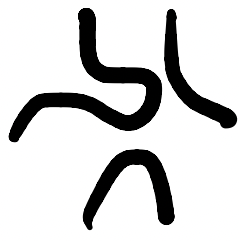}}
\newcommand{\RIIIRioi}{\includegraphics[height=0.25in]{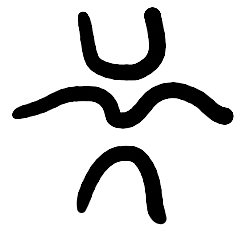}}
\newcommand{\RIIIRiio}{\includegraphics[height=0.25in]{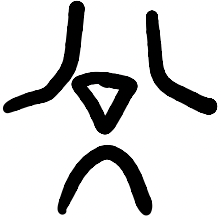}}
\newcommand{\RIIIRiii}{\includegraphics[height=0.25in]{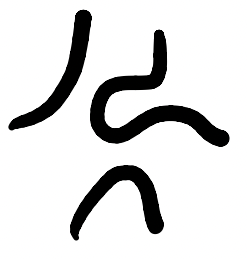}}

\newcommand{\Ciii}{(\C/\C')/\C''}
\newcommand{\Diii}{(\D/\D')/\D''}

\newcommand{\RIIIcase}[6]{
		\begin{tikzpicture}
	 	\node at (0,0) (n1) {#1};
	 	\node at (2,0) (n2) {#2};
	 	\node at (4.5,0) (n3) {$\displaystyle{#3/_{w_+=0}}$};
	 	\node at (7,0) (n4) {#4};
	 	\node at (10,2) (n5) {$\displaystyle{#5/_{w_+=0}}$};
	 	\node at (10,0) (n7) {$+$};
	 	\node at (10,-2) (n6) {#6};
	 	\scriptsize{
  		\node at (0,-1) (blabel) {000};
  		\node at (2,-1) (blabel) {001};
  		\node at (4.5,-1) (blabel) {101};
  		\node at (7,-1) (blabel) {100};
  		\node at (10,1) (blabel) {101};
  		\node at (10,-3) (blabel) {110};
	 	\draw [->] (n1) -- (n2) node[midway,above] {$\partial_0^{Lee}$};
	 	\draw [->] (n1) -- (n2) node[midway,below] {$t_1$};
	 	\draw [->] (n2) -- (n3) node[midway,above] {$\partial_0$};
	 	\draw [->] (n2) -- (n3) node[midway,below] {$\Delta$};
	 	\draw [->] (n3) -- (n4) node[midway,above] {$\Delta^{-1}$};
	 	\draw [->] (n4) -- (n5) node[midway,above left] {$\Delta$};
	 	\draw [->] (n4) -- (n6) node[midway,below left] {$t_2$};
	 	}
	\end{tikzpicture}
	}

\DeclareMathOperator{\even}{even}
\DeclareMathOperator{\odd}{odd}
\DeclareMathOperator{\sgn}{sgn}
\DeclareMathOperator{\Mod}{-mod}

	\allowdisplaybreaks

\counterwithin{theorem}{section}

\author{Champ Davis}
\address{University of Oregon; Department of Mathematics; 108C University Hall; Eugene, OR 97403}
\email{champd@uoregon.edu}

\title{An $L_\infty$-module structure on annular Khovanov homology}

\begin{document}
\bibliographystyle{plain}

\begin{abstract}
Let $L$ be a link in a thickened annulus.  In \cite{GLW17}, Grigsby-Licata-Wehrli showed that the annular Khovanov homology of $L$ is equipped with an action of $\exsltwo$, the exterior current algebra of the Lie algebra $\sltwo$.  In this paper, we upgrade this result to the setting of $L_\infty$-algebras and modules.  That is, we show that $\exsltwo$ is an $L_\infty$-algebra and that the annular Khovanov homology of $L$ is an $L_\infty$-module over $\exsltwo$.  Up to $L_\infty$-quasi-isomorphism, this structure is invariant under Reidemeister moves.  Finally, we include explicit formulas to compute the higher $L_\infty$-operations.
\end{abstract}

\maketitle

\section{Introduction}\label{sec:Intro}

In \cite{K99}, Khovanov defined a bigraded homology group for oriented links in $S^3$ which is a categorification of the Jones polynomial.  Following this, for a compact, oriented surface $\Sigma$, Asaeda, Przytycki, and Sikora introduced a generalization of Khovanov homology for links in $\Sigma \times [0,1]$ that categorifies the Kauffman skein module of $\Sigma$; see \cite{APS04}.  The case where $\Sigma$ is an annulus is known as annular Khovanov homology and has since garnered much attention.  For example, there have been various detection results that have been obtained by exploiting the relationship of annular Khovanov homology with various Floer theories.  In \cite{XY19}, Xie-Zhang use instanton Floer homology to show that annular Khovanov homology detects both the unlink and the closure of the trivial braid.  They also show that it distinguishes braid closures from other links.  More recently, Binns-Martin showed that knot Floer homology detects various torus links, and they used this to show that annular Khovanov homology detects certain braid closures; see \cite{BM20}.

A key feature of annular Khovanov homology is that it is endowed with extra structure not present in ordinary Khovanov homology.  In \cite{GLW17}, Grigsby-Licata-Wehrli show that the annular Khovanov homology of a link is both an $\sltwo$-representation and an $\exsltwo$-representation, where $\exsltwo$ is a $\ZZ$-graded Lie superalgebra related to $\sltwo$.  This structure has been studied in several contexts.  In one direction, Quefflec-Rose generalized this to show that annular Khovanov-Rozansky homology carries an $\mathfrak{sl}_n$-action; see \cite{QR18}.  In another direction, Akhmechet-Krushkal-Willis have made progress towards lifting the $\mathfrak{sl}_2$-action to the stable homotopy refinement of the annular Khovanov homology; see \cite{AKW22}.

In proving that there is an $\sltwo$-representation structure on the annular Khovanov complex $\CKh(L)$, Grigsby-Licata-Wehrli showed that the boundary maps of $\CKh(L)$ commute with the $\sltwo$-action, which shows that the $\sltwo$-action holds at the chain level.  In contrast, the $\exsltwo$-action is well-defined on the annular Khovanov homology $\AKh(L)$, but at the chain level, it only holds up to homotopy.  This observation suggests the existence of an $L_\infty$-module structure on $\AKh(L)$.  In this paper, we exhibit $\exsltwo$ as an $L_\infty$-algebra and upgrade the $\exsltwo$-representation structure to that of an $L_\infty$-module.  This module structure is an invariant of the annular link at both the chain level and on homology.  In particular, we will prove the following theorem.

\begin{theorem}
\label{thm:main}
	Let $L \subset A\times I$ be an annular link.  There is an $L_\infty$-module structure on both $\CKh(L; \ZZ/2\ZZ)$ and $\AKh(L; \ZZ/2\ZZ)$ over the $L_\infty$-algebra $\exsltwo$.  Up to $L_\infty$-quasi-isomorphism, this module structure only depends on the isotopy class of $L$ in $A\times I$.
\end{theorem}

The organization of this paper is as follows.  In section 2, we recall the definitions of $\mathfrak{sl}_2$, $\exsltwo$, and $\exsltwo_{dg}$ and review some key results obtained by Grigsby-Licata-Wehrli.  In section 3, we provide a more detailed background of annular Khovanov homology.  In section 4, we provide relevant background information on $L_\infty$-algebras and modules.  In section 5, we explain how $\exsltwo_{dg}$ and $\exsltwo$ are $L_\infty$-algebras.  In sections 6 and 7, we explain how $\CKh(L)$ and $\AKh(L)$ are $L_\infty$-modules.  In section 8, we prove the invariance of these structures under Reidemeister moves.  In section 9, we provide some examples showing this structure is nontrivial.

\begin{remark}
	The proof of Theorem \ref{thm:main} relies on several results about $L_\infty$-modules.  In particular, the proofs of Theorem \ref{thm:transfer}, Lemma \ref{lem:transfermorphism}, and Theorem \ref{restriction} are given over $\ZZ/2\ZZ$.  We expect these results to hold with signs, but tracking them through their respective proofs is intricate and beyond the scope of this paper.  Outside of these three proofs, we will include signs when appropriate.  Working without signs affects the bracket relations in $\sltwo$, $\exsltwo$, and $\exsltwo_{dg}$; see section \ref{sec:liealgebras}.  The absence of signs also affects the higher operations involved in the $\exsltwo$ $L_\infty$-module structure on $\CKh(L)$; see Theorem \ref{thm:wedgestructure}.

\end{remark}

\subsection{Acknowledgements} The author would like to thank his advisor, Robert Lipshitz, for his guidance and many helpful conversations.

\section{The Lie algebras $\slk{2}, \exsltwo,$ and $\exsltwo_{dg}$}
\label{sec:liealgebras}

In this section, we review the Lie algebras of interest.  We first recall the definition of $\slk{2}$.  Next, we define the Lie superalgebra $\exsltwo$, which will be our main $L_\infty$-algebra of study.  Finally, we define an auxiliary Lie superalgebra, $\exsltwo_{dg}$, which is closely related to $\exsltwo$ and will help us prove several key results.

\subsection{The Lie algebra $\slk{2}$}

To fix notation, we will denote the standard basis for the Lie algebra $\slk{2}$ by $\{e,f,h\}$.  Over $\ZZ$, the Lie bracket relations are given by:
$$
[e,f] = h, \,\, [e,h] = -2e, \,\, [f,h] = 2f.
$$

\subsection{The Lie superalgebra $\exsltwo$}

In \cite{GLW17}, Grigsby-Licata-Wehrli introduce a larger Lie algebra $\exsltwo$ containing $\sltwo$ as a subalgebra.  In fact, $\exsltwo$ has the structure of a $\ZZ$-graded Lie superalgebra.

\begin{definition}
	A \textbf{Lie superalgebra} $\mathfrak{g}$ is a $\ZZ/2\ZZ$-graded vector space $\mathfrak{g}_{\even}\oplus \mathfrak{g}_{\odd}$ equipped with a bilinear map $[\cdot, \cdot]: \g \times \g \to \g$, called the super Lie bracket, satsfying the following conditions:
	\vspace{.5em}

	\begin{tabular}{ll}
		(Super skew-symmetry) &  $[x,y]= -(-1)^{|x||y|}[y,x]$ \\ 
		& \\
		(Super Jacobi identity)& $(-1)^{|x||y|}[x,[y,z]] + (-1)^{|y||x|}[y,[z,x]] + (-1)^{|z||y|}[z,[x,y]]=0$ \\ 
	\end{tabular}
\vspace{.5em}

	Here, $x,y$ and $z$ are homogeneous elements with respect to the $\ZZ/2\ZZ$-grading.  The notation $|x|$ represents the degree of $x$, and the degree of $[x,y]$ is required to be the sum of the degrees of $x$ and $y$, modulo 2.  These conditions should be thought of as analogs of the usual Lie algebra axioms, but with gradings taken into consideration.
\end{definition}

We now describe the exterior current algebra $\exsltwo$ by generators and relations, as presented in \cite{GLW17}.  As vector spaces,
$$
	\exsltwo \cong \sltwo \oplus \sltwo,
$$
where the first summand is in degree 0 and the second in degree 1 with respect to both the $\ZZ$-grading and the $\ZZ/2\ZZ$-grading.  The $\ZZ/2\ZZ$-grading required for the Lie superalgebra structure is the mod 2 reduction of the $\ZZ$-grading.  Denoting the standard basis of the first $\sltwo$ summand by $\{e,f,h\}$ and that of the degree 1 summand by $\{v_{2},v_{-2}, v_0\}$, the bracket relations for the Lie superalgebra $\exsltwo$ are 
\begin{center}
\begin{tabular}[colsep=3em]{l@{\hspace{3em}}l@{\hspace{3em}}l}
	$[e,f] = h$ 					& $[h,e] = 2e$ 						& $[f,v_0] = 2v_{-2}$ \\
	$[e,v_{2}] = 0$ 				& $[h,f] = -2f$ 					& $[f,v_{-2}] = 0$ \\
	$[e,v_0] = -2v_2$ 				& $[h,v_0] = 0$ 					& $[h,v_2] =2v_2$ \\
	$[e,v_{-2}] = v_0 = -[f,v_2]$ 	& $[h,v_{-2}] = -2v_{-2}$ 			& $[v_i,v_j]=0$ for $i,j\in \{2,0,-2\}$. \\
\end{tabular}
\end{center}

\subsection{The Lie superalgebra $\exsltwo_{dg}$}

Following \cite{GLW17}, we describe the $\ZZ$-graded Lie superalgebra $\exsltwo_{dg}$.  As a $\ZZ$-graded super vector space, the degree $0$ generators are $\{e,f,h\}$, and the degree $1$ generators are $\{v_2,v_{-2},d,D\}$.  The defining bracket relations are 
\begin{center}
\vspace{1em}
\begin{tabular}[colsep=3em]{l@{\hspace{3em}}l@{\hspace{3em}}l}
	$[e,f] = h$			&		$[e,v_{-2}] = -[f,v_2]$;		&	$[d,y] = 0$ for all $y\in \{e,f,h,v_2,v_{-2}\}$; \\
	$[h,e] = 2e$;		&		$[f,v_{-2}]= 0$;				&	$[D,y] = 0$ for all $y\in \{e,f,h,v_2,v_{-2}\}$; \\ 
	$[h,f] = -2f$;		&		$[h,v_2]  = 2v_2$;				&	$[d,d] = [D,D] = [v_2,v_2] = [v_{-2},v_{-2}] = 0$. \\
	$[e,v_{2}] = 0$;	&		$[h,v_{-2}] = -2v_{-2}$;		&	$[v_2,v_{-2}] + [d,D] = 0$.	
\end{tabular}
	\vspace{1em}
\end{center}

The structure of $\exsltwo_{dg}$ becomes more clear with the following two lemmas.  The first gives us a basis for $\exsltwo_{dg}$, and the second exhibits $\exsltwo$ as a direct summand of the homology of $\exsltwo_{dg}$ by regarding $\exsltwo_{dg}$ as a chain complex with differential given by the adjoint action of $d$.  Both lemmas are proved in \cite{GLW17}.

\begin{lemma}[\cite{GLW17}; Lemma 6]
Let $\tilde{v_0} = [e,v_{-2}] = -[f,v_2]$, and let $x = [v_2,v_{-2}] = -[d,D] = \frac{1}{2}[\tilde{v}_0,\tilde{v}_0]$.  Then the set $\{e,f,h,v_2,v_{-2},\tilde{v_0},d,D,x\}$ forms a basis of $\exsltwo_{dg}$.
\end{lemma}

\begin{lemma}[\cite{GLW17}; Lemma 7]
\label{lemma:hom}
  The homology of the chain complex $(\exsltwo_{dg}$, $[d,\cdot])$ is isomorphic to the direct sum of $\exsltwo$ and the trivial Lie superalgebra.  That is, $H(\exsltwo_{dg}, [d,\cdot]) \cong \exsltwo \oplus \mathbb{Z}.$
\end{lemma}

\section{Annular Khovanov homology}
\label{sec:annularhomology}

In this section, we review the construction of annular Khovanov homology and recall some of its structure.  For other expositions; see \cite{R13}, \cite{SZ18}, and \cite{GLW17}.  To start, let $L\subset A \times I$ be a link in the thickened annulus.  The link $L$ admits a diagram $P(L) \subset A$ by considering the projection $A\times I \to A \times \{0\}$, and this diagram can be regarded as sitting inside of $S^2 - \{\XX, \OO\}$, where $\XX$ is a basepoint representing the inner boundary of $A$, and $\OO$ is a basepoint representing the outer boundary of $A$; see Figure \ref{AnnularDiagram}.

\begin{figure}[h]
	\centering
	\scalebox{1}{\input{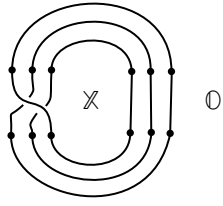}}
	\caption{A diagram $P(L) \subset S^2-\{\XX,\OO\}$ of an annular link $L$, where $\XX$ and $\OO$ represent the inner and outer boundaries of the annulus, respectively.}
	\label{AnnularDiagram}
\end{figure}

If we ignore the basepoint $\XX$, we can form the ordinary Khovanov complex $\CKh(P(L))$.  $\CKh(P(L))$ is generated by oriented Kauffman states, where circles are labeled either $v_+$ or $v_-$.  $\CKh(P(L))$ is also bigraded, where an element of $\CKh^{i,j}(P(L))$ is said to have homological grading $i$ and quantum grading $j$.  Formulas for these gradings are given in \cite{Z18} and \cite{R13}.  

The addition of the basepoint $\XX$ endows $\CKh(P(L))$ with a third grading $k$, called the $k$-grading or the winding-number grading.  For a fixed generator, the associated Kauffman state is a collection of oriented circles, and the $k$-grading is defined to be the algebraic intersection number of this collection of circles with an oriented arc from $\XX$ to $\OO$ that misses all crossings of $P(L)$.  Another way to compute the $k$-grading is to count the number of positively-labeled nontrivial circles and subtract the number of negatively-labeled nontrivial circles, where a nontrivial circle is a circle that separates $\XX$ and $\OO$.  In \cite{R13}, it is proved that the Khovanov differential $\partial$ does not increase the $k$-grading, and so this gives rise to a filtration on $\CKh(P(L))$.  The annular Khovanov homology $\AKh(P(L))$ is the homology of the associated graded object.  Said differently, we can decompose the Khovanov differential as $\partial=\partial_0 + \partial_-$, where $\partial_0$ and $\partial_-$ are the $k$-preserving and $k$-decreasing parts of $\partial$, respectively.  $\AKh(P(L))$ is the homology of the triply-graded chain complex $(\CKh(P(L)), \partial_0)$.  Moreover, up to isomorphism, the annular Khovanov homology does not depend on the diagram $P(L)$ representing $L$, so it makes sense to write $\AKh(L)$.

It is instructive to see how the differential $\partial_0$ of the annular Khovanov complex differs from the usual Khovanov differential $\partial$.  To do so, we need to examine how the $k$-gradings of generators change under merge and split maps.  Denoting trivial circles by T's and nontrivial circles by N's, the three possibilities are TT $\leftrightarrow$ T, NT $\leftrightarrow$ N, and NN $\leftrightarrow$ T; see Figure \ref{MS}.

The formula for the differential $\partial_0$ depends on the types of circles involved, and we list the explicit formulas for each case below.  Recall that trivial circles are labeled by either $w_+$ or $w_-$ and nontrivial circles are labeled by either $v_+$ or $v_-$.

\begin{enumerate}
	\item  When two trivial circles merge into trivial circle, or when a trivial circle splits into two trivial circles:

		    \vspace{.5em}
		    \begin{tabular}{c c}
		\begin{minipage}{0.4\textwidth}
		\begin{center}
			\underline{Merge}\vspace{.5em}

		$\begin{aligned}
			\mathbf{w}_+ &\otimes \mathbf{w}_+ \mapsto \mathbf{w}_+ \\
			\mathbf{w}_+ &\otimes \mathbf{w}_- \mapsto \mathbf{w}_- \\
			\mathbf{w}_- &\otimes \mathbf{w}_+ \mapsto \mathbf{w}_- \\
			\mathbf{w}_- &\otimes \mathbf{w}_- \mapsto 0
		\end{aligned}$
		\end{center}
		\end{minipage} &
		\begin{minipage}{0.4\textwidth}
		\begin{center}
			\underline{Split}\vspace{.5em}

		$\begin{aligned}
			\mathbf{w}_+ &\mapsto \mathbf{w}_+\otimes \mathbf{w}_- + \mathbf{w}_-\otimes \mathbf{w}_+ \\ 
			\mathbf{w}_- &\mapsto \mathbf{w}_-\otimes \mathbf{w}_- \\ 
			& \\ 
			&
		\end{aligned}$
		\end{center}
		\end{minipage}
		\end{tabular}
		\vspace{.5em}

		\item When a trivial circle and a nontrivial circle merge into a nontrivial circle, or when a nontrivial circle splits into a trivial circle and a nontrivial circle:

		\vspace{.5em}
		    \begin{tabular}{c c}
		\begin{minipage}{0.4\textwidth}
		\begin{center}
			\underline{Merge}\vspace{.5em}

		$\begin{aligned}
			\mathbf{w}_+ &\otimes \mathbf{v}_+ \mapsto \mathbf{v}_+ \\
			\mathbf{w}_+ &\otimes \mathbf{v}_- \mapsto \mathbf{v}_- \\
			\mathbf{w}_- &\otimes \mathbf{v}_+ \mapsto 0 \\
			\mathbf{w}_- &\otimes \mathbf{v}_- \mapsto 0
		\end{aligned}$
		\end{center}
		\end{minipage} &
		\begin{minipage}{0.4\textwidth}
		\begin{center}
			\underline{Split}\vspace{.5em}

		$\begin{aligned}
			\mathbf{v}_+ &\mapsto \mathbf{w}_-\otimes \mathbf{v}_+ \\ 
			\mathbf{v}_- &\mapsto \mathbf{w}_-\otimes \mathbf{v}_- \\ 
			& \\ 
			&
		\end{aligned}$
		\end{center}
		\end{minipage}
		\end{tabular}
		\vspace{.5em}

		\item When two nontrivial circles merge into a trivial circle, or when a trivial circle splits into two nontrivial circles:

		\vspace{.5em}
		    \begin{tabular}{c c}
		\begin{minipage}{0.4\textwidth}
		\begin{center}
			\underline{Merge}\vspace{.5em}

		$\begin{aligned}
			\mathbf{v}_+ &\otimes \mathbf{v}_+ \mapsto 0 \\
			\mathbf{v}_+ &\otimes \mathbf{v}_- \mapsto \mathbf{w}_- \\
			\mathbf{v}_- &\otimes \mathbf{v}_+ \mapsto \mathbf{w}_- \\
			\mathbf{v}_- &\otimes \mathbf{v}_- \mapsto 0
		\end{aligned}$
		\end{center}
		\end{minipage} &
		\begin{minipage}{0.4\textwidth}
		\begin{center}
			\underline{Split}\vspace{.5em}

		$\begin{aligned}
			\mathbf{w}_+ &\mapsto \mathbf{v}_+\otimes \mathbf{v}_- + \mathbf{v}_-\otimes \mathbf{v}_+ \\ 
			\mathbf{w}_- &\mapsto 0 \\ 
			& \\ 
			&
		\end{aligned}$
		\end{center}
		\end{minipage}
		\end{tabular}
		\vspace{.5em}

\end{enumerate}

\begin{figure}[H]
		\centering
		\scalebox{1}{\input{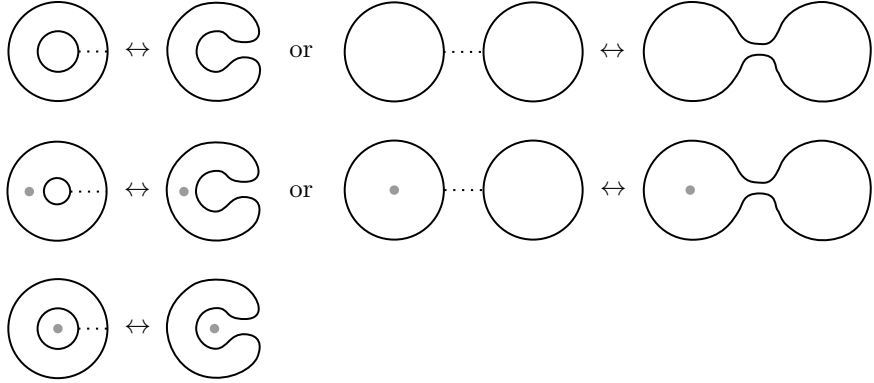}}
		\vspace{1em}
    	\caption{The various ways the operations of merging and splitting along a crossing (indicated by a dashed line) interact with a basepoint. The top illustrates the case of two trivial circles merging into trivial circle (or a trivial circle splitting into two trivial circles).  The middle illustrates the case of a trivial circle and a nontrivial circle merging into a nontrivial circle (or a nontrivial circle splitting into a trivial circle and a nontrivial circle).  The bottom illustrates the case of nontrivial circles merging into a trivial circle (or a trivial circle splitting into two nontrivial circles).} 
    	\label{MS}
\end{figure}

We end this section by briefly describing the $\sltwo$ representation structure on $\AKh(L)$, referring the reader to \cite{GLW17} for details.  Fix a resolution of $P(L)$.  Nontrivial circles, with respect to the basis $\{v_+, v_-\}$, are assigned the 2-dimensional defining representation of $\sltwo$, defined by
\[ h \mapsto \left(\begin{array}{cc} 1 & 0\\
					0 & -1\end{array}\right), \,\, e \mapsto \left(\begin{array}{cc} 0 & 1\\
														0 & 0\end{array}\right), \,\,
				f \mapsto \left(\begin{array}{cc} 0 & 0\\
								1 & 0\end{array}\right).\]
Trivial circles are assigned the 2-dimensional trivial representation.  The resolution is then assigned the tensor product of these representations.  We take the direct sum of all of these representations to obtain the structure of an $\sltwo$-representation on $\CKh(P(L))$.  This action descends to an action on the homology $\AKh(L)$, which Grigsby-Licata-Wehrli then upgrade to an action of $\exsltwo$.  They show that the annular boundary maps commute with the $\sltwo$-action, implying that the $\sltwo$-action holds at the chain level.  In contrast, the $\exsltwo$-action is well-defined on $\AKh(L)$, but at the chain level, it only holds up to homotopy \cite{GLW17}.  This observation leads us to consider this situation in terms of $L_\infty$-algebras and modules.

\section{$L_\infty$-algebras and modules}

For the reader's convenience, we review some basic definitions from the theory of $L_\infty$-algebras and modules.  More details can be found in a survey article by Stasheff in \cite{S18}.  Another good overview of the subject is given in \cite{KS22}.  For a more graphical approach to the definitions and results in this section; see \cite{D22}.

\subsection{Preliminaries}

\begin{definition}
		Let $\sigma \in S_n$ be a permutation.  If $X$ is a set, then $\sigma$ induces a map $\sigma^{\bullet}: X^n \to X^n$, defined by $\sigma^{\bullet}(x_1, x_2, \ldots, x_n) = (x_{\sigma(1)}, \ldots, x_{\sigma(n)})$.  If $X$ is a vector space, $\sigma$ induces a similarly-defined map on the $n$-fold tensor product $\sigma^{\bullet}: X^{\otimes n} \to X^{\otimes n}$.
	 \end{definition} 

	 \begin{definition}
	 	Let $i_1, i_2, \ldots, i_r$ be non-negative integers, with $i_1 + i_2 + \cdots + i_r = n$.  A permutation $\sigma\in S_n$ is called an \textbf{$(i_1, i_2, \ldots, i_r)$-unshuffle} if the restriction of $\sigma$ to the integers in the interval $[i_1+\cdots + i_l + 1, i_l + \cdots + i_{l+1}]$ is order-preserving for all $0\leq l \leq r-1$.  We will denote the set of $(i_1, i_2, \ldots, i_r)$-unshuffles in $S_n$ by $S(i_1, \ldots, i_r)$. 
	 \end{definition} 

	 \begin{definition}
	 	 We will denote by $S'(i_1, \ldots, i_r)$ the set of $(i_1, i_2, \ldots, i_r)$-unshuffles $\sigma$ in $S_n$ satisfying $i_1 \leq i_2 \leq \cdots \leq i_r$ and $\sigma(i_1 + \cdots + i_{l-1} + 1) < \sigma(i_1 + \cdots + i_l+1)$ if $i_l = i_{l+1}$.
	 \end{definition}

	 \begin{definition}
	 	We will denote by $\overline{S}(i_1,\ldots, i_r)$ the set of $(i_1, i_2, \ldots, i_r)$-unshuffles $\sigma$ in $S'(i_1,\ldots,i_r)$ satisfying $\sigma(1)=1$.
	 \end{definition}

	 \begin{definition}
	 	Let $V$ be a graded vector space.  For $\sigma \in S^n$ and $v_i \in V$, let $\epsilon(\sigma):= \epsilon(\sigma,v_1, \ldots, v_n)$ be the total Koszul sign of $\sigma$.  To compute $\epsilon(\sigma)$, every time two elements of degrees $x$ and $y$ are transposed, we record a sign of $(-1)^{xy}$, and $\epsilon(\sigma)$ is the total product of such signs.   Define $\mathbf{\chi(\sigma)}:=\epsilon(\sigma)\sgn(\sigma)$ to be the product of the Koszul sign and the sign of the permutation $\sigma$.
	 \end{definition}

	 \begin{remark}
	 	Let $f: A\to B$ and $g: C\to D$ be graded maps of graded algebras.  We will also follow the Kozsul sign convention of including a sign in the evaluation of the map $f\otimes g$.  That is, for an element $x\otimes y \in A\otimes C$, 
	 	$$(f\otimes g)(x\otimes y) = (-1)^{|x||g|}f(x) \otimes g(y).$$
	 \end{remark}

\begin{definition}
	Let $V$ be a graded vector space.  An \textbf{$L_\infty$-algebra structure} on $V$ is a collection of skew-symmetric multilinear maps $\{l_k: V^{\otimes k} \to V\}$ of degree $k-2$.  That is, each $l_k$ is skew-symmetric in the sense that
	$$l_k\circ \sigma^{\bullet}(x_1,x_2,\ldots,x_k) = \chi(\sigma) l_k(x_1, x_2, \ldots, x_k)$$
	for all $\sigma\in S_k$ and $x_i\in V$.  These maps also must satsify the generalized Jacobi identity:
	$$\sum_{i+j=n+1}\sum_{\sigma} \chi(\sigma) (-1)^{i(j-1)} l_j \circ (l_i \otimes \Id) \circ \sigma^{\bullet} = 0$$
	Here, $i\geq 1$, $j\geq 1$, $n\geq 1$, and the inner summation is taken over all $(i,n-i)$-unshuffles.
\end{definition}

\begin{remark}
	This definition follows the chain complex convention.  If instead our $L_\infty$-algebra is a cochain complex, we require each $l_k$ to have degree $2-k$.  There are similar cohain complex conventions for the following definitions.
\end{remark}

	\begin{definition}
		Let $(L,l_i)$ and $(L', l_i')$ be $L_\infty$-algebras.  An \textbf{$L_\infty$-algebra homomorphism} from $L$ to $L'$ is a sequence of skew-symmetric multilinear maps $\{f_n: L^{\otimes n} \to L'\}$ of degree $n-1$ such that 
		\begin{align*}
			\sum_{j+k = n+1} \sum_{\sigma\in S(k,n-k)} \epsilon_1 \cdot f_j \circ (l_k \otimes \Id) \circ \sigma^{\bullet}  +  \sum_{\substack{\tau \in S'(i_1,\ldots,i_r) \\  i_1+\ldots+i_r=n}} \epsilon_2 \cdot l_r' \circ (f_{i_1} \otimes \cdots \otimes f_{i_r})\circ \tau^{\bullet} = 0
		\end{align*}
		where $\epsilon_1 = \chi(\sigma) (-1)^{k(j-1)+1}$ and $\epsilon_2 = \chi(\tau) (-1)^{\frac{r(r-1)}{2} + \sum_{s=1}^{r-1} i_s(t-s)}$.
		
	\end{definition}

\begin{definition}
	Let $(L,l_k)$ be an $L_\infty$-algebra.  The data of an \textbf{$L_\infty$-module} over $L$ consists of a graded vector space $M$, together with skew-symmetric multilinear maps $\{k_n: L^{\otimes n-1} \otimes M \to M \mid 1\leq n<\infty\}$ of degree $n-2$ satisfying:
	\begin{align*}
			&\sum_{\substack{p+q=n+1 \\ p<n}}\sum_{\sigma(n)=n} \epsilon_1 \cdot k_q\circ (l_p \otimes \Id) \circ \sigma^{\bullet}  +\sum_{p+q=n+1} \sum_{\sigma(p)=n} \epsilon_2\epsilon_3 \cdot k_q\circ \delta^{\bullet} \circ ( k_p\otimes \Id) \circ \sigma^{\bullet} = 0
		\end{align*}
		 where $\epsilon_1 = \epsilon_2 = \chi(\sigma)(-1)^{p(q-1)}$ and $\sigma$ is a $p$-unshuffle in $S_n$.  In the case of $\sigma(p)=n$, we used the skew-symmetry of $k_q$ and introduced $\delta^{\bullet}$ to permute the $k_p$ term past the remaining elements to ensure that $k_q: L^{\otimes q-1}\otimes M \to M$.  Explicitly,
		  \begin{align*}
		 k_q\big( \underbrace{k_p(x_{\sigma(1)}, \ldots, x_{\sigma(p)})}_{\in M}, x_{\sigma(p+1)}, \ldots, x_{\sigma(n)} \big) &= \epsilon_3 \cdot k_q\Big( \delta^{\bullet} \big( \underbrace{k_p(x_{\sigma(1)}, \ldots, x_{\sigma(p)})}_{\in M}, x_{\sigma(p+1)}, \ldots, x_{\sigma(n)} \big) \Big) \\
		 &=\epsilon_3 \cdot k_q\big( x_{\sigma(p+1)}, \ldots, x_{\sigma(n)}, \underbrace{k_p(x_{\sigma(1)}, \ldots, x_{\sigma(p)})}_{\in M} \big)
		 \end{align*}
		 where $\epsilon_3 = \chi(\delta) = (-1)^{q-1}(-1)^{(p+\sum_{s=1}^p |x_{\sigma(s)}|)(\sum_{s=p+1}^{n}|x_{\sigma(s)}|)}$.
\end{definition}

\begin{definition}
	Following \cite{A14}, let $(L, l_i)$ be an $L_\infty$-algebra, and let $(M,k_i)$ and $(M',k_i')$ be $L_\infty$-modules over $L$.  An \textbf{$L_\infty$-module homomorphism} from $M$ to $M'$ is a collection of skew-symmetric multilinear maps  $\{h_n: L^{\otimes{(n-1)}} \otimes M \to M'\}$ of degree $n-1$ satisfying:
	\begin{align*}
		\sum_{\substack{i+j=n+1 \\ i<n}} \sum_{\sigma(n)=n} \epsilon_1 \cdot h_j \circ (l_i \otimes \Id) \circ \sigma^{\bullet}
		& +\sum_{i+j=n+1} \sum_{\sigma(i)=n} \epsilon_2 \cdot h_j \circ \delta^{\bullet} \circ (k_i \otimes \Id) \circ \sigma^{\bullet} \\
	 &+\sum_{r+s=n+1} \sum_{\tau} \epsilon_3 \cdot k_r'\circ (\Id\otimes h_s) \circ (\tau^{\bullet} \otimes \Id)=0
	\end{align*}
	where $\epsilon_1 = \epsilon_2 = \chi(\sigma)(-1)^{i(j-1)+1}$ and $\epsilon_3 = \chi(\tau)(-1)^{(s-1)(\sum_{t=1}^{n-s} x_{\tau(t)})}$, $\sigma$ is an $i$-unshuffle in $S_n$, and $\tau$ is an $(n-s)$-unshuffle in $S_{n-1}$.  Similar to the definition of $L_\infty$-module, we include the permutation $\delta$ to ensure the module element is in the correct location.
\end{definition}

\subsection{Transfer of Structure}
\label{sec:transfer}

It is possible to use an existing $L_\infty$-algebra or $L_\infty$-module to obtain a new $L_\infty$-structure on a particular chain complex.  In this paper, we will use chain contractions to transfer $L_\infty$-structures, and we will also make use of the restriction of scalars functor. 

\begin{definition}
	Let $(A,d_A)$ and $(B,d_B)$ be chain complexes.  A \textbf{chain contraction} from $A$ onto $B$ consists of two chain maps $q: A\to B$ and $i: B\to A$ of degree 0, together with a homotopy $K: A \to A$ of degree 1.  That is, we have the following diagram.
	$$
		\begin{tikzcd}[row sep = normal, column sep=normal]
			 A \arrow[loop left, distance=1em, "K"] \arrow[r, shift left=1, "q"] & B  \arrow[l, shift left=1, "i"]
		\end{tikzcd}
	$$
	These maps $q,i$, and $K$ must satisfy the following conditions:
	\begin{align*}
		& q \circ i = \Id_B \quad \mbox{ and } \quad \Id_A - i\circ q = K\circ d_A + d_A \circ K  \\ 
		& K^2 = K\circ i = q\circ K = 0
	\end{align*}
	We will denote a chain contraction by $(A,B,i,q,K)$.
\end{definition}

\begin{remark}
	If $(A,d_A)$ and $(B,d_B)$ are cochain complexes, we require $|K|=-1$.
\end{remark}

If $L$ is an $L_\infty$-algebra and $L'$ is a chain complex, formulas exist in the literature for how to transfer the $L_\infty$-algebra structure from $L$ to $L'$, given a chain contraction $(L,L',i,q,K)$.   Following \cite[Theorem 1]{MMF22}, the chain maps $i$ and $q$ can also be extended to $L_\infty$-algebra homomorphisms $I: L' \to L$ and $Q: L\to L'$ such that $Q\circ I = \Id_{L'}$.  The transferred $L_\infty$-algebra structure on $L'$ is unique up to quasi-isomorphism, and the formula for the transferred bracket $\{l'_k\}$ can be given inductively as follows.  Set $K\theta_1=-i$ and define $\theta_n: (L')^{\otimes n} \to L$ for $n\geq 2$ by 
$$\theta_n(x_1, \ldots, x_n)= \sum_{k=2}^n \sum_{\substack{\sigma \in \overline{S}(i_1, \ldots, i_k) \\ i_1+\cdots+i_k=n\\i_1\leq \cdots \leq i_k}} \epsilon_1 \cdot l_k(I_{i_1} \otimes \cdots \otimes I_{i_k})\circ \sigma^\bullet(x_1, \ldots, x_n)$$
where $\epsilon_1$ is given by the Koszul sign convention.  Then for all $n\geq 2$, we define $l'_n = q\circ \theta_n$ and $I_n = K\circ \theta_n$.  

We can also use chain contractions to transfer an $L_\infty$-module structure.  We will make use of this technique in the proof of the invariance of the $\exsltwo_{dg}$ $L_\infty$-module structure under Reidemeister moves.

\begin{theorem}
\label{thm:transfer}
	Let $L$ be an $L_\infty$-algebra, and let $M$ be an $L_\infty$-module over $L$.  Given a chain contraction 
	$$
		\begin{tikzcd}[row sep = normal, column sep=normal]
			 M \arrow[loop left, distance=1em, "T"] \arrow[r, shift left=1, "q"] & M'  \arrow[l, shift left=1, "i"]
		\end{tikzcd}
	$$
	then $M'$ inherits the structure of an $L_\infty$-module over $L$, with transferred bracket given by
	$$k_n':= \sum_{\substack{\tau \in S(i_1, \ldots, i_t) \\ i_1 + \cdots + i_t = n-1}} q \circ A_t \circ (\tau^\bullet \otimes i)$$
	where $A_t: L^{\otimes i_1} \otimes \cdots \otimes L^{\otimes i_t} \otimes M \to M$ is defined inductively as follows.  Let $A_1 = k_{i_1+1}$ and define $A_{t} = A_1 \circ \delta_2^\bullet \circ [(T\circ A_{t-1}) \otimes \Id] \circ \delta_1^\bullet$, where $i_1, \ldots, i_t$ are positive integers; see Figure \ref{fig:transferdef}.  
\end{theorem}

\begin{figure}[h!]
		\centering
		\scalebox{1}{\input{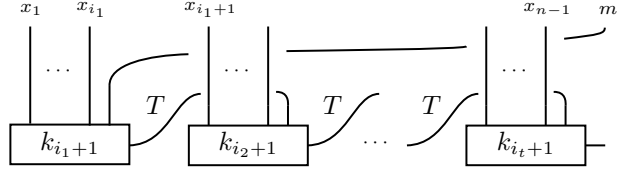}}

		\caption{A graphical depiction of the map $A_t$.}
		\label{fig:transferdef}
	\end{figure}

\begin{remark}
	The permutations $\delta_i$ in the definition of $A_t$ above are required to ensure that the module element is the last input of each $k_{i_r+1}: L^{\otimes r} \otimes M \to M$.  Explicitly, $\delta_i$ is the unique permutation so that $\delta_i^\bullet$ shifts the module element to the required position and preserves the order of the other elements.  For example, in Figure \ref{fig:transferdef}, $\delta_1$ is the permuation 
	$$\delta_1 = \begin{pmatrix}
		1 & \cdots & i_1 & i_1+1 & i_1+2 &  \cdots & n \\
		1 & \cdots & i_1 & n & i_1+1 & \cdots & n-1
	\end{pmatrix}$$
	Throughout the proof of Theorem \ref{thm:transfer}, we will make use of similar permutations $\lambda_i$ to correctly place the module element while preserving the order of the remaining elements.  We will not write down these permutations explicitly, but they can be readily determined by examining the figures in the appendix.
\end{remark}

\begin{remark}
	We remind the reader that we are ignoring signs in the above theorem and that the result is proved over a field of characteristic two.
\end{remark}
\begin{proof}
	We must show that the above definition for $k_n'$ satisfies the $L_\infty$-module relation:

	\begin{align*}
			&\sum_{\substack{p+q=n+1 \\ p<n}}\sum_{\sigma(n)=n} k'_q\circ (l_p \otimes \Id) \circ \sigma^{\bullet} = \sum_{p+q=n+1} \sum_{\sigma(p)=n} k'_q\circ \lambda^{\bullet} \circ ( k'_p\otimes \Id) \circ \sigma^{\bullet} 
		\end{align*}

	The idea of the proof is as follows.  Start by replacing the $k'_q$ and $k'_p$ terms using the definition of $k'_n$.  Next, apply the $L_\infty$-module relation for $k_n$ to the terms involving $l_p$ on the left-hand side.  After that, use fact that $\Id_{M} - i\circ q = k_1T + Tk_1$ to replace terms on the right-hand side.  Terms will then cancel in pairs.  Graphical representations of the formulas in this proof are provided in the appendix.

	\vspace{1em}
	\noindent \textbf{Step 1.} Focusing on the left-hand side of the $L_\infty$-module relation, we can replace $k_n'$ using its definition to obtain the following sum.
	$$\sum_{\substack{p+q = n+1 \\ p< n}} \sum_{\sigma(n)=n} \sum_{\substack{\tau \in S(i_1, \ldots, i_t) \\ i_1 + \cdots + i_t =  n-p}} q \circ A_t \circ (\tau^\bullet \otimes i) \circ (l_p \otimes \Id) \circ \sigma^\bullet$$

	\vspace{1em}
	\noindent \textbf{Step 2.} We can combine $\sigma$ and $\tau$ into $\eta$ and $\psi$.  Since $\tau$ is an unshuffle, the $l_p$ term will be the first element in some block, which we denote by $i_l$.  Defining $s=p+i_l - 1$, we obtain the following sum.
	\begin{align*}
		\sum_{p=1}^{n-1} \sum_{\substack{\eta \in S(i_1, \ldots, p+i_l-1, \ldots, i_t) \\ i_1 + \cdots + i_t =  n-p \\ 1 \leq t \leq n-p \\ 1\leq l \leq t}} \sum_{\psi \in S(p,s-p)} q \circ A_t \circ \left[ \Id \otimes ((l_p \otimes \Id) \circ \psi^\bullet) \otimes \Id \right] \circ (\eta^\bullet \otimes i)
	\end{align*}

	\vspace{1em}
	\noindent \textbf{Step 3.}  The goal now is to unpack the $A_t$ terms using the definition of $A_t$ in order to apply the $L_\infty$-module relation.  Because $A_t$ only makes sense for $t\geq 1$, we break up the sum into several cases.  In the first case, the $l_p$ term is in the first box.  In the second case, the $l_p$ term is somewhere in the middle, in which case we need at least three boxes.  In the third case, the $l_p$ term is in the last box.  Note further that the only way for there to be one box is if $p=n-1$.  We obtain the following sum.

\begin{align*}
		&\sum_{p=1}^{n-2} \  \sum_{\substack{\eta \in S(p+i_1-1, i_2, \ldots, i_t)\\ i_1 + \cdots + i_t=n-p \\ 2 \leq t \leq n-p \\ l=1}} \sum_{\psi \in S(p,s-p)}  q \circ A_{t-1} \circ \lambda^\bullet_3 \circ (T \otimes \Id) \circ \bigg[ \left[k_{i_l+1} \circ (l_p \otimes \Id) \circ \psi^\bullet \right] \otimes \Id \bigg] \\ 
		& \hspace{10em} \circ \lambda_1^\bullet \circ (\eta^\bullet \otimes i)\\
		&+ \sum_{p=1}^{n-3} \ \sum_{\substack{\eta \in S(i_1, \ldots, p+i_l-1, \ldots, i_t)\\ i_1 + \cdots + i_t=n-p \\ 3 \leq t \leq n-p \\ 2 \leq l \leq t-1}} \sum_{\psi \in S(p,s-p)}  q \circ A_{t-l} \circ \lambda^\bullet_3 \circ (T \otimes \Id) \circ \bigg[ \left[k_{i_l+1} \circ (l_p \otimes \Id) \circ \psi^\bullet \right] \otimes \Id \bigg] \\ 
		&\hspace{10em}  \circ \lambda_2^\bullet \circ \left[ (T\circ A_{l-1}) \otimes \Id\right] \circ \lambda_1^\bullet \circ (\eta^\bullet \otimes i)\\
		&+ \sum_{p=1}^{n-2} \ \sum_{\substack{\eta \in S(i_1, \ldots, i_{t-1}, p+i_t-1)\\ i_1 + \cdots + i_t=n-p \\ 2 \leq t \leq n-p \\ l=t}} \sum_{\psi \in S(p,s-p)}  q \circ  \bigg[ \left[k_{i_l+1} \circ (l_p \otimes \Id) \circ \psi^\bullet \right] \otimes \Id \bigg]   \circ \lambda_2^\bullet \circ \left[ (T\circ A_{t-1}) \otimes \Id\right] \\ 
		&\hspace{10em} \circ \lambda_1^\bullet \circ (\eta^\bullet \otimes i)\\
		&+ \sum_{s=n-1} \sum_{p=1}^{s} \ \sum_{\psi \in S(p,s-p)}  q \circ \bigg[ \left[k_{i_l+1} \circ (l_p \otimes \Id) \circ \psi^\bullet \right] \otimes \Id \bigg] \circ (\Id \otimes i)
	\end{align*}

We now reindex over the size of $s=p+i_l-1$.

\begin{align*}
		&\sum_{s=1}^{n-2} \sum_{p=1}^{s} \ \sum_{\substack{\eta \in S(s, i_2, \ldots, i_t)\\ i_1 + \cdots + i_t=n-p \\ 2 \leq t \leq n-p \\ l=1}}  \sum_{\psi \in S(p,s-p)}  q \circ A_{t-1} \circ  \lambda^\bullet_3 \circ (T \otimes \Id) \circ \bigg[ \left[k_{i_l+1} \circ (l_p \otimes \Id) \circ \psi^\bullet \right] \otimes \Id \bigg] \\ 
		&\hspace{10em} \circ \lambda_1^\bullet \circ (\eta^\bullet \otimes i) \\
		&+ \sum_{s=1}^{n-3} \sum_{p=1}^{s} \ \sum_{\substack{\eta \in S(i_1, \ldots, s, \ldots, i_t)\\ i_1 + \cdots + i_t=n-p \\ 3 \leq t \leq n-p \\ 2 \leq l \leq t-1}}  \sum_{\psi \in S(p,s-p)} q \circ A_{t-l} \circ  \lambda^\bullet_3 \circ (T \otimes \Id) \circ \bigg[ \left[k_{i_l+1} \circ (l_p \otimes \Id) \circ \psi^\bullet \right] \otimes \Id \bigg] \\ 
		&\hspace{10em}  \circ \lambda_2^\bullet \circ \left[ (T\circ A_{l-1}) \otimes \Id\right] \circ \lambda_1^\bullet \circ (\eta^\bullet \otimes i)\\
		&+ \sum_{s=1}^{n-2} \ \sum_{p=1}^{s} \sum_{\substack{\eta \in S(i_1, \ldots, i_{t-1},s)\\ i_1 + \cdots + i_t=n-p \\ 2 \leq t \leq n-p \\ l=t}} \sum_{\psi \in S(p,s-p)}  q \circ  \bigg[ \left[k_{i_l+1} \circ (l_p \otimes \Id) \circ \psi^\bullet \right] \otimes \Id \bigg] \\ 
		&\hspace{10em}  \circ \lambda_2^\bullet \circ \left[ (T\circ A_{t-1}) \otimes \Id\right] \circ \lambda_1^\bullet \circ (\eta^\bullet \otimes i) \\ 
		&+ \sum_{s=n-1} \sum_{p=1}^{s} \ \sum_{\psi \in S(p,s-p)}  q \circ \bigg[ \left[k_{i_l+1} \circ (l_p \otimes \Id) \circ \psi^\bullet \right] \otimes \Id \bigg] \circ (\Id \otimes i)
	\end{align*}
We can now apply the $L_\infty$-module relation.

	\begin{align*}
		 & \sum_{s=1}^{n-2} \sum_{p=1}^{s+1} \ \sum_{\substack{\eta \in S(s, i_2, \ldots, i_t)\\ i_1 + \cdots + i_t=n-p \\ 2 \leq t \leq n-p \\ l=1}}  \sum_{\psi \in S(p-1,s-p+1)}  q \circ A_{t-l} \circ \lambda^\bullet_3 \circ (T \otimes \Id) \circ \bigg[ \left[k_{s-p+2} \circ (k_p \otimes \Id) \circ \psi^\bullet \right] \otimes \Id \bigg] \\ 
		 &\hspace{10em}  \circ \lambda_1^\bullet \circ (\eta^\bullet \otimes i)\\
		&+  \sum_{s=1}^{n-3} \sum_{p=1}^{s+1} \ \sum_{\substack{\eta \in S(i_1, \ldots, s, \ldots, i_t)\\ i_1 + \cdots + i_t=n-p \\ 3 \leq t \leq n-p \\ 2 \leq l \leq t-1}}  \sum_{\psi \in S(p-1,s-p+1)} q \circ A_{t-l} \circ  \lambda^\bullet_3 \circ (T \otimes \Id) \circ \bigg[ \left[k_{s-p+2} \circ (k_p \otimes \Id) \circ \psi^\bullet \right] \otimes \Id \bigg] \\ 
		&\hspace{10em}  \circ \lambda_2^\bullet \circ \left[ (T\circ A_{l-1}) \otimes \Id\right] \circ \lambda_1^\bullet \circ (\eta^\bullet \otimes i)\\
		&+  \sum_{s=1}^{n-2} \ \sum_{p=1}^{s+1} \sum_{\substack{\eta \in S(i_1, \ldots, i_{t-1},s)\\ i_1 + \cdots + i_t=n-p \\ 2 \leq t \leq n-p \\ l=t}} \sum_{\psi \in S(p-1,s-p+1)} q \circ  \lambda^\bullet_3 \circ (T \otimes \Id) \circ \bigg[ \left[k_{s-p+2} \circ (k_p \otimes \Id) \circ \psi^\bullet \right] \otimes \Id \bigg] \\ 
		&\hspace{10em}  \circ \lambda_2^\bullet \circ \left[ (T\circ A_{t-1}) \otimes \Id\right] \circ \lambda_1^\bullet \circ (\eta^\bullet \otimes i)\\
		&+  \sum_{s=n-1} \sum_{p=1}^{s+1} \ \sum_{\psi \in S(p-1,s-p+1)}  q \circ  \bigg[ \left[k_{s-p+2} \circ (k_p \otimes \Id) \circ \psi^\bullet \right] \otimes \Id \bigg] \circ (\Id \otimes i)
	\end{align*}

	\vspace{1em}
	\noindent \textbf{Step 4.}  Combine $\psi$ and $\eta$ into $\kappa$, and reintroduce $A_t$ into the notation, treating the cases $p=1$ and $p=s+1$ separately.  Indeed, we observe that for $1 < p < s+1$, we may combine both the $k_p$ and $k_{s-p+2}$ terms into an $A_t$ term.  Otherwise, we will have a $k_1$ term.

	\begin{align*}
		 & \sum_{s=1}^{n-2} \sum_{p=1} \ \sum_{\substack{\kappa \in S(s, i_2, \ldots, i_t)\\ i_1 + \cdots + i_t=n-p \\ 2 \leq t \leq n-p \\ l=1}}   q \circ A_{t} \circ \lambda^\bullet_3 \circ (k_1 \otimes \Id) \circ \lambda_1^\bullet \circ (\kappa^\bullet \otimes i)\\
		  &+\sum_{s=1}^{n-2} \sum_{p=2}^{s} \ \sum_{\substack{\kappa \in S(p-1,s-p+1, i_2, \ldots, i_t)\\ i_1 + \cdots + i_t=n-p \\ 2 \leq t \leq n-p \\ l=1}}   q \circ A_{t} \circ \lambda^\bullet_3 \circ (T \otimes \Id) \circ (A_1 \otimes \Id) \circ \lambda_1^\bullet \circ (\kappa^\bullet \otimes i)\\
		  &+\sum_{s=1}^{n-2} \sum_{p=s+1} \ \sum_{\substack{\kappa \in S(s, i_2, \ldots, i_t)\\ i_1 + \cdots + i_t=n-p \\ 2 \leq t \leq n-p \\ l=1}}   q \circ A_{t-1} \circ \lambda^\bullet_3 \circ ((T\circ k_1) \otimes \Id) \circ (A_1 \otimes \Id) \circ \lambda_1^\bullet \circ (\kappa^\bullet \otimes i)\\
		  &+\sum_{s=1}^{n-3} \sum_{p=1} \ \sum_{\substack{\kappa \in S(i_1, \ldots, i_{l-1},s,i_{l+1}, \ldots, i_t)\\ i_1 + \cdots + i_t=n-p \\ 3 \leq t \leq n-p \\ 2\leq l\leq t}}   q \circ A_{t-l+1} \circ \lambda^\bullet_3 \circ ((k_1 \circ T) \otimes \Id) \circ (A_{l-1}\otimes \Id) \circ \lambda_1^\bullet \circ (\kappa^\bullet \otimes i)\\
		  &+\sum_{s=1}^{n-3} \sum_{p=2}^{s} \ \sum_{\substack{\kappa \in S(i_1, \ldots, i_{l-1},p-1,s-p+1,i_{l+1}, \ldots, i_t)\\ i_1 + \cdots + i_t=n-p \\ 3 \leq t \leq n-p \\ 2\leq l\leq t}}   q \circ A_{t-l+1} \circ \lambda^\bullet_3 \circ A_{l}\otimes \Id) \circ \lambda_1^\bullet \circ (\kappa^\bullet \otimes i)\\
		  &+\sum_{s=1}^{n-3} \sum_{p=s+1} \ \sum_{\substack{\kappa \in S(i_1, \ldots, i_{l-1},s,i_{l+1}, \ldots, i_t)\\ i_1 + \cdots + i_t=n-p \\ 3 \leq t \leq n-p \\ 2\leq l\leq t}}   q \circ A_{t-l} \circ \lambda^\bullet_3 \circ ((T \circ k_1) \otimes \Id) \circ (A_{l}\otimes \Id) \circ  \lambda_1^\bullet \circ (\kappa^\bullet \otimes i)\\
		  &+\sum_{s=1}^{n-2} \sum_{p=1} \ \sum_{\substack{\kappa \in S(i_1, \ldots, i_{t-1},s)\\ i_1 + \cdots + i_t=n-p \\ 2 \leq t \leq n-p \\ l= t}}   q \circ A_{1} \circ \lambda^\bullet_3 \circ ((k_1 \circ T) \otimes \Id) \circ (A_{t-1}\otimes \Id) \circ \lambda_1^\bullet \circ (\kappa^\bullet \otimes i)\\
		  &+\sum_{s=1}^{n-2} \sum_{p=2}^{s} \ \sum_{\substack{\kappa \in S(i_1, \ldots, i_{t-1},p-1,s-p+1)\\ i_1 + \cdots + i_t=n-p \\ 2 \leq t \leq n-p \\ l=t}}   q \circ A_{1} \circ \lambda^\bullet_3 \circ (A_{t}\otimes \Id) \circ \lambda_1^\bullet \circ (\kappa^\bullet \otimes i)\\
		  &+\sum_{s=1}^{n-2} \sum_{p=s+1} \ \sum_{\substack{\kappa \in S(i_1, \ldots, i_{t-1},s)\\ i_1 + \cdots + i_t=n-p \\ 2 \leq t \leq n-p \\ l=t}}   q \circ k_1 \circ \lambda^\bullet_3 \circ (A_t \otimes \Id) \circ \lambda_1^\bullet \circ (\kappa^\bullet \otimes i)\\
		 &+  \sum_{s=n-1} \sum_{p=1} \ \sum_{\kappa=\Id}  q \circ k_{s+1} \circ (k_1 \otimes \Id) \circ \lambda_1^\bullet \circ (\Id \otimes i)\\
		 &+  \sum_{s=n-1} \sum_{p=2}^{s} \ \sum_{\kappa \in S(p-1,s-p+1)}  q \circ k_{s-p+2} \circ (k_p \otimes \Id) \circ \lambda_1^\bullet \circ (\kappa^\bullet \otimes i)\\
		 &+  \sum_{s=n-1} \sum_{p=s+1} \ \sum_{\kappa =\Id}  q \circ k_1 \circ (k_{s+1} \otimes \Id) \circ \lambda_1^\bullet \circ (\kappa^\bullet \otimes i)\\
	\end{align*}

	\vspace{1em}
	\noindent \textbf{Step 5.} We can combine the sums above.  The first term is obtained by combining terms 4 and 7 above.  The second term is obtained by combining terms 3 and 6 above.  The third term is obtained by combining terms 2, 5, 8, and 11 above.  The fourth term is obtained by combining terms 1 and 10.  The last term is obtained by combining terms 9 and 12 above.

	\begin{align*}
		&\sum_{s=1}^{n-2} \sum_{\substack{\kappa \in S(i_1, \ldots, i_{l-1},s, i_{l+1} \ldots, i_t)\\  2\leq t \leq n-1 \\ 2 \leq l \leq t \\ i_1 + \cdots + i_t=n-1 \\ i_l=s}}   q \circ A_{t-l+1} \circ \lambda^\bullet_2 \circ \big[(k_1 \circ T \circ A_{l-1}) \otimes \Id \big] \circ \lambda_1^\bullet \circ (\kappa^\bullet \otimes i)\\
		&+\sum_{s=1}^{n-2} \sum_{\substack{\kappa \in S(i_1, \ldots, i_{l-1},s, i_{l+1} \ldots, i_t)\\  2\leq t \leq n-1-s \\ 1 \leq l \leq t-1 \\ i_1 + \cdots + i_t=n-1-s \\ i_l=0}}  q \circ A_{t-l} \circ \lambda^\bullet_2 \circ \big[(T\circ k_1 \circ A_l) \otimes \Id\big] \circ \lambda_1^\bullet \circ (\kappa^\bullet \otimes i)\\
		&+\sum_{s=1}^{n-1} \ \sum_{p=2}^{s} \sum_{\substack{\kappa \in S(i_1, \ldots, i_{l-1}, p-1,s-p+1, i_{l+1}, \ldots, i_t)\\ 1\leq t \leq n-p \\ 1\leq l\leq t \\ i_1 + \cdots + i_t=n-p \\ i_l=s+1-p}}   q \circ A_{t-l+1} \circ \lambda^\bullet_2 \circ \big[(\Id \otimes A_l) \otimes \Id \big] \circ \lambda_1^\bullet \circ (\kappa^\bullet \otimes i)\\
		&+\sum_{s=1}^{n-1} \ \sum_{\substack{\kappa \in S(s, i_2, \ldots, i_t) \\ 1 \leq t \leq n-1 \\ l=1\\i_1 + \cdots + i_t=n-1 \\  i_l=s}}   q \circ A_{t} \circ \lambda^\bullet_3 \circ (k_1 \otimes \Id) \circ \lambda_1^\bullet \circ (\kappa^\bullet \otimes i)\\
		&+ \sum_{s=1}^{n-1} \ \sum_{\substack{\kappa \in S(i_1, \ldots, i_{t-1},s) \\ 1 \leq t \leq n-1-s \\ l=t \\ i_1 + \cdots + i_t=n-1-s \\  i_l=0}}    q \circ k_1 \circ \lambda^\bullet_3 \circ (A_t \otimes \Id) \circ \lambda_1^\bullet \circ (\kappa^\bullet \otimes i)
	\end{align*}

	\vspace{1em}
	\noindent \textbf{Step 6.}  Change notation.  In the first sum, let $c_1, \ldots, c_w$ be $i_1, \ldots, i_{l-1}$ and $d_1, \ldots, d_x$ be $s,i_{l+1},\ldots, i_t$.  The conditions $t\geq 2$ and $2\leq l \leq t$ imply that $w,x\geq 1$.  Make similar changes to the other sums.  In the second sum, let $c_1, \ldots, c_w$ be $i_1, \ldots, i_{l-1},s$ and $d_1, \ldots, d_x$ be $i_{l+1},\ldots, i_t$, and in the third sum let $c_1, \ldots, c_w$ be $i_1, \ldots, i_{l-1},p-1$ and $d_1, \ldots, d_x$ be $s-p+1,i_{l+1},\ldots, i_t$.

	\begin{align*}
	&\sum_{\substack{\kappa \in S(c_1, \ldots, c_w, d_1 \ldots, d_x) \\ c_1 + \cdots + c_w + d_1 + \cdots + d_x=n-1 \\ w,x\geq 1}}   q \circ A_{x} \circ \lambda^\bullet_2 \circ \big[(k_1 \circ T \circ A_{w}) \otimes \Id \big] \circ \lambda_1^\bullet \circ (\kappa^\bullet \otimes i)\\
	&+\sum_{\substack{\kappa \in S(c_1, \ldots, c_w, d_1 \ldots, d_x) \\ c_1 + \cdots + c_w + d_1 + \cdots + d_x=n-1 \\ w,x\geq 1}}   q \circ A_{x} \circ \lambda^\bullet_2 \circ \big[(T\circ k_1 \circ A_w) \otimes \Id\big] \circ \lambda_1^\bullet \circ (\kappa^\bullet \otimes i)\\
	&+\sum_{\substack{\kappa \in S(c_1, \ldots, c_w, d_1 \ldots, d_x) \\ c_1 + \cdots + c_w + d_1 + \cdots + d_x=n-1 \\ w,x\geq 1}}   q \circ A_{x} \circ \lambda^\bullet_2 \circ \big[(\Id \otimes A_w) \otimes \Id \big] \circ \lambda_1^\bullet \circ (\kappa^\bullet \otimes i)\\
	&+ \sum_{\substack{\kappa \in S(d_1 \ldots, d_x) \\ d_1 + \cdots + d_x=n-1 \\ x\geq 1}}    q \circ A_{x} \circ \lambda^\bullet_3 \circ (k_1 \otimes \Id) \circ \lambda_1^\bullet \circ (\kappa^\bullet \otimes i)\\
	&+ \sum_{\substack{\kappa \in S(c_1, \ldots, c_w) \\ c_1 + \cdots + c_w = n-1 \\ w \geq 1}}    q \circ k_1 \circ \lambda^\bullet_3 \circ (A_w \otimes \Id) \circ \lambda_1^\bullet \circ (\kappa^\bullet \otimes i)
	\end{align*}

	\vspace{1em}
	\noindent \textbf{Step 7.}  Focusing now on the right-hand side, we substitute for $k_n'$ using its definition.  We consider the cases $p=1$ and $q=1$ separately, and use the fact that $k_1' \circ q = q\circ k_1$ and $k_1 \circ i = i\circ k_1'$, since $i$ and $q$ are chain maps.
	\begin{align*}
	&\sum_{p=2}^{n-1} \ \sum_{\substack{\sigma(p)=n}} \   \sum_{\substack{\alpha \in S(a_1, \ldots, a_r)\\ a_1+\cdots+a_r=p-1}} \sum_{\substack{\beta \in S(b_1, \ldots, b_s)\\ b_1+\cdots+b_s=q-1}} q \circ A_s \circ (\beta^\bullet \otimes i) \circ \lambda^\bullet \circ \left[ (q \circ A_r \circ (\alpha^\bullet \otimes i)) \otimes \Id \right] \circ \sigma^\bullet\\
	&+\sum_{p=1} \ \sum_{\substack{\sigma(p)=n}} \  \sum_{\substack{\beta \in S(b_1, \ldots, b_s)\\ b_1+\cdots+b_s=n-1}}  q \circ A_s \circ (\beta^\bullet \otimes i)  \circ \lambda^\bullet \circ \left[ (k_1 \circ i) \otimes \Id \right] \circ \sigma^\bullet\\
	&+\sum_{p=n} \ \sum_{\substack{\sigma(p)=n}} \   \sum_{\substack{\alpha \in S(a_1, \ldots, a_r)\\ a_1+\cdots+a_r=n-1}}  q \circ k_1 \circ (\beta^\bullet \otimes i) \circ \lambda^\bullet \circ \left[ q \circ A_r \circ (\alpha^\bullet \otimes i) \otimes \Id \right] \circ \sigma^\bullet
	\end{align*}

	\vspace{1em}
	\noindent \textbf{Step 8.}  We can combine $\sigma, \alpha$, and $\beta$ into one unshuffle $\theta$.
	\begin{align*}
	&\sum_{p=2}^{n-1} \   \sum_{\substack{\theta \in S(a_1, \ldots, a_r, b_1, \ldots, b_s)\\ a_1+\cdots+a_r=p-1 \\ b_1 + \cdots + b_s = q-1}}  q \circ A_s \circ \lambda^\bullet_2 \circ \left[ (i\circ q \circ A_r) \otimes \Id \right] \circ \lambda_1^\bullet \circ (\theta^\bullet \otimes i)\\
	&+\sum_{p=1} \  \sum_{\substack{\theta \in S(b_1, \ldots, b_s)\\ b_1+\cdots+b_s=n-1}}  q \circ A_s \circ \lambda^\bullet_2 \circ (k_1 \otimes \Id) \circ \lambda_1^\bullet \circ (\theta^\bullet \otimes i)\\
	&+\sum_{p=n} \   \sum_{\substack{\theta \in S(a_1, \ldots, a_r)\\ a_1+\cdots+a_r=n-1}}  q \circ k_1 \circ \lambda^\bullet_2 \circ (A_r \otimes \Id) \circ \lambda_1^\bullet \circ (\theta^\bullet \otimes i)
	\end{align*}

	\noindent \textbf{Step 9.}  Use the fact that $\Id_{M} - i\circ q = k_1\circ T + T\circ k_1$.

	\begin{align*}
	&\sum_{p=2}^{n-1} \   \sum_{\substack{\theta \in S(a_1, \ldots, a_r, b_1, \ldots, b_s)\\ a_1+\cdots+a_r=p-1 \\ b_1 + \cdots + b_s = q-1}}  q \circ A_s \circ \lambda^\bullet_2 \circ \left[ (k_1 \circ T \circ A_r) \otimes \Id \right] \circ \lambda_1^\bullet \circ (\theta^\bullet \otimes i)\\
	&+\sum_{p=2}^{n-1} \   \sum_{\substack{\theta \in S(a_1, \ldots, a_r, b_1, \ldots, b_s)\\ a_1+\cdots+a_r=p-1 \\ b_1 + \cdots + b_s = q-1}}  q \circ A_s \circ \lambda^\bullet_2 \circ \left[ (T \circ k_1 \circ A_r) \otimes \Id \right] \circ \lambda_1^\bullet \circ (\theta^\bullet \otimes i)\\
	&+\sum_{p=2}^{n-1} \   \sum_{\substack{\theta \in S(a_1, \ldots, a_r, b_1, \ldots, b_s)\\ a_1+\cdots+a_r=p-1 \\ b_1 + \cdots + b_s = q-1}}  q \circ A_s \circ \lambda^\bullet_2 \circ \left[ (\Id \circ A_r) \otimes \Id \right] \circ \lambda_1^\bullet \circ (\theta^\bullet \otimes i)\\
	&+\sum_{p=1} \  \sum_{\substack{\theta \in S(b_1, \ldots, b_s)\\ b_1+\cdots+b_s=n-1}}  q \circ A_s \circ \lambda^\bullet_2 \circ (k_1 \otimes \Id) \circ \lambda_1^\bullet \circ (\theta^\bullet \otimes i)\\
	&+\sum_{p=n} \   \sum_{\substack{\theta \in S(a_1, \ldots, a_r)\\ a_1+\cdots+a_r=n-1}}  q \circ k_1 \circ \lambda^\bullet_2 \circ (A_r \otimes \Id) \circ \lambda_1^\bullet \circ (\theta^\bullet \otimes i)
	\end{align*}

	These terms are precisely the terms in Step 6, and so the terms cancel in pairs.  Hence $L_\infty$-module relation holds for $k'_n$.
\end{proof}

\begin{lemma}
\label{lem:transfermorphism}
	In the setting of Theorem \ref{thm:transfer}, if $M'$ has the $L_\infty$-module structure induced by a chain contraction, then the map $i: M' \to M$ can be extended to an $L_\infty$-module homomorphism, where $I_1 = i$, and for $n\geq 2$, we define $I_n$ by
	$$I_n:= \sum_{\substack{\tau \in S(i_1, \ldots, i_t) \\ i_1 + \cdots + i_t = n-1}} T \circ A_t \circ (\tau^\bullet \otimes i)$$
	The map $A_t: L^{\otimes i_1} \otimes \cdots \otimes L^{\otimes i_t} \otimes M \to M$ is defined as in the statemen of \ref{thm:transfer}; see Figure \ref{fig:transferdef}.
\end{lemma}
\begin{proof}(Sketch).
	We will prove that $I$ satisfies the $L_\infty$-module homomorphism relation for $n=2$.  Indeed, for $x\in L$ and $m\in M'$ we must show that
	\begin{equation}
		I_2(l_1(x),m) + I_2(x,k'_1(m)) + I_1(k'_2(x,m)) = k_2(x_1,I_1(m)) + k_1(I_2(x,m))
	\end{equation}
	Working on the left-hand side and substituting in the definitions of $I$ and $k'$, we get
	\begin{equation}
		T\circ k_2(l_1(x),i(m))+T\circ k_2(x_1,i\circ k'_1(m))+i\circ q\circ k_2(x,i(m))
	\end{equation}
	Next, we use the fact that $i$ is a chain map and that $i\circ q = T\circ k_1 + k_1 \circ T + \Id_M$, to see that (2) is equal to
	\begin{equation}
		T\circ k_2(l_1(x),i(m))+T\circ k_2(x_1,k_1(i(m)))+T \circ k_1 \circ k_2(x,i(m))+ k_1 \circ T \circ k_2(x,i(m)) + k_2(x,i(m))
	\end{equation}
	Applying the $L_\infty$-module relation to the first two terms in (3), we obtain
	\begin{equation}
		T\circ k_1 \circ k_2(x,i(m))+T \circ k_1 \circ k_2(x,i(m))+ k_1 \circ T \circ k_2(x,i(m)) + k_2(x,i(m))
	\end{equation}
	Now, the first two terms cancel, and what remains is $k_2(x_1,I_1(m)) + k_1(I_2(x,m))$, as desired.  The proof of the general case is similar to the proof of Theorem \ref{thm:transfer}
\end{proof}

Another way to obtain a new $L_\infty$-module structure is by using the restriction of scalars functor, which is a functor that relates the categories of modules of two $L_\infty$-algebras.

\begin{theorem}\cite[Theorem 1]{D22}.
\label{restriction}
		Let $L,L'$ be $L_\infty$-algebras over $\FF_2$ and $I: L' \to L$ be an $L_\infty$-algebra homomorphism.  Then there exists a restriction of scalars functor $I^*: L\Mod \to L'\Mod$.
\end{theorem}

In particular, if $(M,k)$ is an $L$-module, then $I^*M:=(M,k')$ is an $L'$-module with induced bracket $k'_n: L^{\otimes n-1} \otimes M \to M$ that is defined by
	\begin{align*}
				k_n' &=\sum_{r=1}^{n-1} \sum_{\substack{\tau\in S'(i_1,\ldots,i_r) \\ i_1+\ldots+i_r=n-1}} k_{r+1} \circ (I_{i_1} \otimes \cdots \otimes I_{i_r} \otimes \Id) \circ (\tau^\bullet \otimes \Id)
	\end{align*}

Moreover, if $f: M\to N$ is a homomorphism of $L_\infty$-modules, $(I^*f)_1=f_1$.  Therefore, if $M$ and $N$ are quasi-isomorphic, so too are $I^*M$ and $I^*N$.

\section{The $L_\infty$-algebra structure on $\exsltwo$}
\label{sec:algebras}

Since $\exsltwo_{dg}$ is a Lie superalgebra, it is an $L_\infty$-algebra with no higher operations.  In this section, we will use a cochain contraction to transfer this $L_\infty$-algebra structure on $\exsltwo_{dg}$ to $\exsltwo$, and then we will show that all higher operations in the $L_\infty$-algebra structure on $\exsltwo$ vanish.

\begin{lemma}
	\label{contraction}
	There exist maps $i$ and $q$ so that the data 
	\begin{equation}
	\label{eqn:contract}
	\tag{$*$}
		\begin{tikzcd}[row sep = normal, column sep=normal]
			 \exsltwo_{dg} \arrow[loop left, distance=1em, "K"] \arrow[r, shift left=1, "q"] & H(\exsltwo_{dg})  \arrow[l, shift left=1, "i"]
		\end{tikzcd}
	\end{equation}
	satisfies the definition of a cochain contraction.
\end{lemma}
\begin{proof}
	By Lemma \ref{lemma:hom}, $H(\exsltwo_{dg}) \cong \exsltwo \oplus \ZZ$ with basis $\{v_2, v_{-2}, v_0, d, e,f,h\}$.  Writing out the basis elements of $\exsltwo_{dg}$ and $H(\exsltwo_{dg})$, with their degrees above them, we have
		\begin{center}
		\begin{tabular}{cccc}
			\ & 2 & 1 & 0 \\
			$\exsltwo_{dg}$ & $x$ & $v_2, v_{-2}, \tilde{v}_0, d, D$ & $e,f,h$ \\ 
			$H(\exsltwo_{dg})$ & 0 & $v_2, v_{-2}, \tilde{v}_0, d$ & $e,f,h$ \\ 
		\end{tabular}
		\end{center}
	The maps $i$ and $q$ are easy to define.  Let $i$ lift every element to its corresponding element in $\exsltwo_{dg}$, and let $q$ be the projection back down, sending $x$ and $D$ to 0.  Define the chain homotopy $K$ to be 0 for every element except for $x$, in which case we define $K(x) = -D$.

	It is straightforward to check that $i$ and $q$ are chain maps.  The differential in $\exsltwo$ is 0, so $i\partial = 0$.  Also, $\partial i = 0$, since the elements in the image of $i$ are in the kernel of $[d,\cdot]$, which is the differential in $\exsltwo_{dg}$.  On the other hand, $\partial q = 0$.  Also, $q\partial = 0$, since the only element in the image of $[d,\cdot]$ is $x$, which is sent to 0 by $q$.  It is also straightforward to check that all of the chain contraction conditions are satisfied.  
	\end{proof}

\begin{lemma}
	The Lie superalgebra $H(\exsltwo_{dg})$ inherits an $L_\infty$-algebra structure induced by \eqref{eqn:contract}, and this $L_\infty$-algebra structure has no higher operations. 
\end{lemma}
\begin{proof}
	Following section \ref{sec:transfer}, the formulas for the transfered bracket tell us that
	\begin{align*}
			I_n &= \sum_{j=1}^{n-1}\sum_{\sigma \in \overline{S}(j,n-j)} \epsilon(\sigma) \cdot K \circ l_2 \circ (I_j\otimes I_{n-j}) \circ \sigma_\bullet \\
			l_n' &= \sum_{j=1}^{n-1}\sum_{\sigma \in \overline{S}(j,n-j)} \epsilon(\sigma) \cdot q \circ l_2 \circ (I_j\otimes I_{n-j}) \circ \sigma_\bullet \\
		\end{align*}

	Recall that $I_1 = -i$.  For $n=2$, the only unshuffle in $\overline S(1,1)$ is the identity.  So, $I_2(x_1,x_2) = K(l_2(I_1(x_1),I_1(x_2)))$.  Since $K(x) = -D$ and is 0 otherwise, 
	\begin{align*}
		I_2(v_{2},v_{-2}) &= K(x) = -D \\
		I_2(v_{-2},v_{2}) &= K(x) = -D \\ 
		I_2(\tilde{v}_0, \tilde{v}_0) &= K(2x) = -2D \\
		I_2(x_1, x_2) &= 0 \mbox{ otherwise }
	\end{align*}

	Moreover, $l_2'(x_1, x_2) = q(l_2(I_1(x_1), I_1(x_2)))$, and so to compute the bracket of two elements in $H(\exsltwo_{dg})$, we lift them to $\exsltwo_{dg}$, take their bracket in $\exsltwo_{dg}$, and then quotient back to $H(\exsltwo_{dg})$.  

	Now, let $n\geq 3$.  For all $m\geq 2$, $I_m$ is in the image of $K$, and so $I_m(x_1, \ldots, x_m)=cD$ for some scalar $c$.  But then for any $1 \leq j \leq n-1$ and $\sigma \in \overline{S}(j,n-j)$, $q\circ l_2 \circ (I_j\otimes I_{n-j}) \circ \sigma_\bullet$ is 0, since either the $l_2 \circ (I_j\otimes I_{n-j}) \circ \sigma_\bullet$ term is 0, as $[D,y] = 0$ for all $y\in \{e,f,h,v_2,v_{-2}, \tilde{v}_0, D\}$, or $q$ will send this term to 0 since the only nonzero bracket involving $D$ is $[d,D]=-x$, and $q(x) =0$.  Hence $l_n' = 0$ for $n\geq 3$, and so the Lie superalgebra $H(\exsltwo_{dg})$ has no higher operations. 
\end{proof}

	\begin{theorem}
		The Lie superalgebra $\exsltwo$ inherits an $L_\infty$-algebra structure as a subalgebra of $H(\exsltwo_{dg})$, and this $L_\infty$-algebra structure has no higher operations. 
	\end{theorem}
	\begin{proof}
		The map $H(\exsltwo_{dg}) \to \exsltwo$ that sends $e,f,h, v_2, v_{-2}, \tilde{v}_0, d$ to $e,f,h, v_2, v_{-2}, v_0, 0$ is surjective with 1-dimensional kernel.  So $H(\exsltwo_{dg}) \cong \exsltwo \oplus \ZZ$, where the $\ZZ$ summand is generated by the element $d$.  But the bracket of $d$ with everything in the $\exsltwo$ summand is 0, so $\exsltwo$ is a direct sum not only as a vector space, but also as an $L_\infty$-algebra.  So $\exsltwo$ is an $L_\infty$-algebra as a subalgebra of $H(\exsltwo_{dg})$.
	\end{proof}

\section{The $L_\infty$-module structure on $\CKh(L)$}
\label{sec:CKhmodule}

Viewing $\exsltwo_{dg}$ and $\exsltwo$ as $L_\infty$-algebras, in this section we will exhibit $\CKh(P(L))$ as an $L_\infty$-module over $\exsltwo$.  Also, fix a diagram $P(L)$ of the annular link $L$.  We will simplify notation and write $\CKh(L)$ and $\AKh(L)$ instead of $\CKh(P(L))$ and $\AKh(P(L))$.

\begin{theorem}
\label{thm:wedgestructure}
	Let $L$ be an annular link and $m\in \CKh(L)$.  Then $\CKh(L)$ is an $L_\infty$-module over the $L_\infty$-algebra $\exsltwo$.  One of the higher operations is given in terms of the Lee differential: $k_3(v_{2},v_{-2},m) = \partial_0^{Lee}(m)$.  In particular, the $L_\infty$-module structure is nontrivial if $\partial_0^{Lee}: \CKh(L) \to \CKh(L)$ is nonzero.
\end{theorem}
\begin{proof}
	To start, $\CKh(L)$ is an $L_\infty$-module over $\exsltwo_{dg}$, where the $k_2$ operation is given by the usual module action, and $k_n=0$ for $n\geq 3$.  The module actions of elements of the basis $\{e,f,h,v_2,v_{-2},\tilde{v}_0, d,D,x\}$ are as follows.  The actions of $e,f,h$ were described at the end of section \ref{sec:annularhomology}, and $v_2$, $v_{-2}$, $d$, and $D$ act by $\partial_+^{Lee}$, $\partial_-$, $\partial_0$, and $\partial_0^{Lee}$, respectively; see \cite{GLW17}.  The actions of $\tilde{v}_0$ and $x$ can then be determined by the bracket relations.

	Now, we have a cochain contraction from $\exsltwo_{dg}$ onto its homology, so we can transfer the $\exsltwo_{dg}$-module structure to obtain a new module structure over $H(\exsltwo_{dg})$.  We can then restrict this module structure to the copy of $\exsltwo$ that sits inside of $H(\exsltwo_{dg})$.

	To see that the induced module structure is nontrivial, recall the cochain contraction from Lemma \ref{contraction}.
	$$
		\begin{tikzcd}[row sep = normal, column sep=normal]
			 \exsltwo_{dg} \arrow[loop left, distance=1em, "K"] \arrow[r, shift left=1, "q"] & H(\exsltwo_{dg})  \arrow[l, shift left=1, "i"]
		\end{tikzcd}
	$$
	Examining the restriction of scalars formulas from Theorem \ref{restriction}, we see that
	$$
		k_3'(x_1,x_2,m)=k_{3}(I_{1}(x_1),I_1(x_2),m) - k_2(I_2(x_1,x_2),m) = - k_2(I_2(x_1,x_2),m)
	$$
	for $x_1, x_2 \in H(\exsltwo_{dg})$ and $m\in \CKh(L)$.  Here, $k_n$ is the $L_\infty$-module operation for $\exsltwo_{dg}$, and recall that $k_n=0$ for $n\geq 3$.  Since $I_2(v_{2}, v_{-2}) = -D$ and $I_2(v_{-2}, v_{2}) = -D$, and since $D$ acts by $\partial_0^{Lee}$, we conclude that
	$k'_3(v_{2}, v_{-2}, m) = \partial_0^{Lee}(m)$ and $k'_3(v_{-2}, v_{2}, m) = \partial_0^{Lee}(m)$,
	 showing that we obtain higher operations.
\end{proof}

\section{Reidemeister Moves}
\label{sec:reidemeister}

\subsection{Invariance of the $\exsltwo_{dg}$-module structure}

In this section, we follow Khovanov's original proof that Khovanov homology is invariant under Reidemeister moves; see \cite{K99}.  There, Khovanov constructs quasi-isomorphisms between a given Khovanov complex and the complex obtained after applying a particular Reidemeister move.  Here, we upgrade these quasi-isomorphisms to $\exsltwo_{dg}$ $L_\infty$-module quasi-isomorphisms.

\begin{theorem}
	The $L_\infty$-module structure on $\CKh(L)$ is invariant under Reidemeister I.
\end{theorem}

\begin{proof}
	Let $\RIc$ and $\RIa$ denote the annular chain complexes before and after applying an RI move, respectively.  Our goal is to construct a quasi-isomorphism of $L_\infty$-modules $\{h_n\}: \RIa \to \RIc$.  Because the $\exsltwo_{dg}$-module structures on these complexes have no higher operations, it suffices to give a quasi-isomorphism $h_1: \RIa \to \RIc$ that respects the module action, since we can then take $h_n = 0$ for $n\geq 2$.  To this end, let $\C$ be the complex 
	$$\C := \RIa = \RIb \xrightarrow{m} \RIc\{1\}$$ and let $\C'$ be the subcomplex
	$$\C' := \RIb_{w_+}\xrightarrow{m} \RIc\{1\},$$ where $\RIb_{w_+}$ means that the extra circle is labeled $w_+$.  A straightforward check of the actions of the basis elements $\{e,f,h,v_{-2}, v_2, \tilde{v}_0, d,D,x\}$ on $\C'$ shows that $\C'$ is an $sl_2(\wedge)_{dg}$-submodule.  Moreover, $\C'$ is acyclic, since we can write $\C'$ as the mapping cone of the isomorphism $m$.
			\[\C' = \RIb_{w_+}\xrightarrow{m} \RIc\{1\} = 
				\begin{tikzcd}[row sep = tiny, column sep=small]
					\cdots \arrow[r] & \RIb_{w_+} \arrow[r, "d"]\arrow[ddr, "m"] & \RIb_{w_+} \arrow[r] & \cdots \\ 
					 & \oplus & \oplus  &  \\ 
					\cdots \arrow[r] & \RIc\{1\} \arrow[r, "d"] & \RIc\{1\} \arrow[r] & \cdots
				\end{tikzcd}
			\]
	Therefore, the quotient complex $\C/\C'$ is the complex $\RIb/_{w_+=0}\rightarrow 0$, and it is isomorphic to $\RIc$ as chain complexes via the map $z\otimes w_- \mapsto z$.  To summarize, we have constructed a chain map $\RIa \to \RIc$ given by 
	$$y\otimes w_+ + z\otimes w_- + x \mapsto z\otimes w_- \mapsto z$$
	for $y,z \in \C(*0)$ and $x \in \C(*1)$ (we have labeled the crossing formed by the Reidemeister I move last in the chain complex), and this map induces an isomorphisms on homology
	$$H(\RIa) = H(\C) \cong H(\C/\C') \cong H(\RIc)$$

	To complete the proof, we need to check that this composition respects the $\exsltwo_{dg}$ action.  Certainly the first map does, as it is the quotient map of an $\exsltwo_{dg}$-submodule.  For the second map, if $s\in \exsltwo_{dg}$, mapping over and then acting by $s$ gives $z\otimes w_- \mapsto s z$.  On the other hand, acting first by $s$ and then mapping over gives 
			$
				s(z\otimes w_-) = sz\otimes w_- \mapsto sz
			$.
	\end{proof}

\begin{theorem}
	The $L_\infty$-module structure on $\CKh(L)$ is invariant under Reidemeister II.
\end{theorem}
\begin{proof} There is a more direct way to prove RII invariance, but the method that follows will be useful in proving RIII invariance.  Consider the diagrams in Figure \ref{RII-2}.

\begin{figure}[H]
\def\entry#1{{\parbox{1in}{\centering $#1$}}}
\begin{eqnarray*}
 \begin{array}{c}\xymatrix{
    \entry{\RIIb\{1\}} \ar[r]^m
      \ar@{}[rd]|{\displaystyle \C}
    & \entry{\RIIc\{2\}} \\
    \entry{\RIId} \ar[u]^\Delta \ar[r]
    & \entry{\RIIe\{1\}} \ar[u]
  }\end{array}
  & \supset & 
  \begin{array}{c}\xymatrix{
    \entry{\RIIb_{w_+}\{1\}} \ar[r]^m
      \ar@{}[rd]|{\displaystyle \C'}
    & \entry{\RIIc\{2\}} \\
    \entry{0} \ar[u] \ar[r]
    & \entry{0} \ar[u]
  }\end{array}
  \end{eqnarray*}
  
\end{figure}

\begin{figure}[H]
\def\entry#1{{\parbox{1in}{\centering $#1$}}}
\begin{eqnarray*}
  \begin{array}{c}\xymatrix{
    \entry{\RIIb\{1\}_{/ w_+=0}} \ar[r] & \entry{0} \\
     \entry{\RIId} \ar[u]^\Delta \ar[r]^{d_{\star 0}} & \entry{\RIIe\{1\}} \ar[u]
  }\\\C/\C'
  \end{array} 
  & \supset & 
  \begin{array}{c}
    \xymatrix{
      \entry{\beta} \ar[r] \ar[rd]^{\tau=d_{\star 0}\Delta^{-1}}
      & \entry{0} \\
      \entry{\alpha} \ar[u]^\Delta \ar[r]^{d_{\star 0}}
      & \entry{\tau\beta} \ar[u]
    } \\ \C'''
  \end{array}
\end{eqnarray*}
\end{figure}

\begin{figure}[H]
\def\entry#1{{\parbox{1in}{\centering $#1$}}}
\begin{eqnarray*}
  \begin{array}{c}
    \xymatrix{
      \entry{\beta} \ar[r]
        \ar[rd]^{\beta=\tau\beta}
      & \entry{0} \\
      \entry{0} \ar[u] \ar[r]
      & \entry{\gamma} \ar[u]
    } \\
    (\C/\C')/\C'''
  \end{array}
\end{eqnarray*}
\caption{The relevant complexes in the proof of RII invariance.  A similar diagram appears in \cite{BN02}.}
  \label{RII-2}
\end{figure}

As complexes, the composition
$$\RIIa = \C \xrightarrow{q} \C/\C' \xrightarrow{p} (\C/\C')/\C''' \xrightarrow{f} \RIIe$$
is a chain of quasi-isomorphisms; see \cite{BN02}.  Our goal is to show that these complexes are actually quasi-isomorphic as $L_\infty$-modules.  Since $\C'''$ is not an $L_\infty$-submodule, we do not immediately have an $L_\infty$-module structure on $(\C/\C')/\C'''$.  Our strategy then will be to give chain contractions from $\C$ to $\C/\C'$ and from $\C/\C'$ to $(\C/\C')/\C'''$ in order to equip these quotients with $L_\infty$-module structures.  Doing so will give us our desired $L_\infty$-module quasi-isomomorphisms.  To this end, define $i: \C/\C' \to \C$ to be the map
	$$i(z) = \begin{cases}
		z-m^{-1}\partial_{\C}(z) & \mbox{if } \mbox{$z$ is in the top left} \\
		0 & \mbox{if } \mbox{$z$ is in the top right}\\
		z & \mbox{if } \mbox{$z$ is in the bottom left}\\
		z-m^{-1}\partial_{\C}(z) & \mbox{if } \mbox{$z$ is in the bottom right}
	\end{cases}$$ 
	where the map $m^{-1}: \C \to \C$ is zero except on the top right vertex.  There, it will be the inverse to the isomorphism that merges a circle with the small circle labeled $w_+$.  
\begin{remark}
	The map $i$ above takes an element $z \in \C/\C'$ and views it as an element of $\C$.  The complex $\C$ has a preferred basis of Khovanov generators, and $\C/\C'$ has a preferred basis consisting of basis elements of $\C$ not in $\C'$.  So, before applying $i$, we should apply a map $i_0: \C/\C' \rightarrow \C$ as $\FF_2$ vector spaces, but we will suppress this for brevity.
\end{remark}
	Now, if $K: \C \to \C$ is the map 
	$$K(z) = \begin{cases}
		0 & \mbox{if } \mbox{$z$ is in the top left} \\
		m^{-1}(z) & \mbox{if } \mbox{$z$ is in the top right}\\
		0 & \mbox{if } \mbox{$z$ is in the bottom left}\\
		0 & \mbox{if } \mbox{$z$ is in the bottom right}
	\end{cases}$$
	 the data
	\[
		\begin{tikzcd}[row sep = normal, column sep=normal]
			 \C \arrow[loop left, distance=1em, "K"] \arrow[r, shift left=1, "q"] & \C/\C'  \arrow[l, shift left=1, "i"]
		\end{tikzcd}
	\]
	satisfies the requirements of a chain contraction, which we can use to transfer the $L_\infty$-module structure from $\C$ to $\C/\C'$.  In particular, since $i$ was a quasi-isomorphism of chain complexes, we obtain a quasi-isomorphism of $L_\infty$-modules $I_n: \C/\C' \to \C$, where $\C/\C'$ has the induced $L_\infty$-module structure.  In fact, there are no higher operations on $\C/\C'$.  After examining the formula for the induced operation, this follows from the fact that $\C$ itself has no higher operations, that the image of $K$ is in $\C'$, and that $\C'$ is an $L_\infty$-submodule of $\C$.

	Next, since every nonzero element of $(\C/\C')/\C'''$ is equivalent to some element $\gamma$ in the bottom right, we can define $j(\gamma) := \gamma$, thought of as an element of $\C/\C'$. Then if $H: \C/\C' \to \C/\C'$ is the map 
	$$H(z) = \begin{cases}
		\Delta^{-1}(z) & \mbox{if } \mbox{$z$ is in the top left} \\
		0 & \mbox{if } \mbox{$z$ is in the top right}\\
		0 & \mbox{if } \mbox{$z$ is in the bottom left}\\
		0 & \mbox{if } \mbox{$z$ is in the bottom right}
	\end{cases}$$
	 the data
	\[
		\begin{tikzcd}[row sep = normal, column sep=normal]
			 \C/\C' \arrow[loop left, distance=1em, "H"] \arrow[r, shift left=1, "p"] & (\C/\C')/\C'''  \arrow[l, shift left=1, "j"]
		\end{tikzcd}
	\]
	also satisfies the requirements of a chain contraction.  We obtain a quasi-isomorphism of $L_\infty$-modules $J_n: \C/\C' \to \C$, where $(\C/\C')/\C'''$ has the induced $L_\infty$-module structure from $\C/\C'$.  There are no higher operations on $(\C/\C')/\C'''$ as well.  To see this, note that because $\C/\C'$ has no higher operations, the induced module operation on $(\C/\C')/\C'''$ is of the form
	$k_n'= \sum_{\substack{\tau \in S(1, \ldots, 1) \\ i_1= \cdots = i_{n-1} = 1}} q \circ A_{n-1} \circ (\tau^\bullet \otimes i)$; see Figure \ref{fig:induced}.

	\begin{figure}[h]

		\centering

		\scalebox{1}{\input{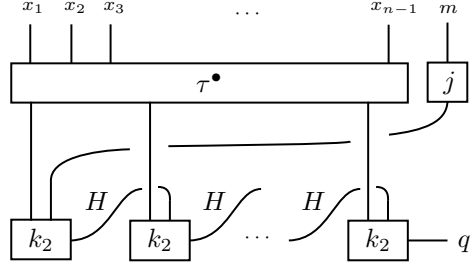}}

		\caption{The transferred bracket on $(\C/\C')/\C'''$.  Here, the labeled edges represent the application of that particular map.  For example, $k_3'(x_1,x_2,m) = q\circ k_2(x_1,H\circ k_2(x_2,j(m))) + q\circ k_2(x_2,H\circ k_2(x_1,j(m)))$}.
		\label{fig:induced}
	\end{figure}

	Since the image of $j$ is concentrated in the bottom right corner of $\C/\C'$, and $H$ is zero everywhere except the top-left, it follows that all higher operations vanish.  As for the module operation $k_2$, an element $s \in \exsltwo_{dg}$ acts on $\gamma \in (\C/\C')/\C'''$ by
	\begin{align*}
		s\cdot \gamma &= p(s\cdot j(\gamma)) = p(q(s\cdot (i\circ j(\gamma)))) = (p\circ q)(s\cdot \gamma-s\cdot m^{-1}\partial_{\C}(\gamma))
	\end{align*}
	That is, we consider the difference $s\cdot \gamma-s\cdot m^{-1}\partial_{\C}(\gamma)$ as an element of $\C$, and then quotient twice.  It remains to show that the degree shift map $f: (\C/\C')/\C''' \to \RIIe$ respects this action, that is, $f(s\cdot \gamma) = s\cdot f(\gamma)$ for $s\in \exsltwo_{dg}$ and $\gamma \in (\C/\C')/\C'''$.  We compute that
	\begin{align*}
	 	f(s \cdot \gamma) - s\cdot f(\gamma) &= f((p\circ q)(s\cdot \gamma)- (p \circ q)(s\cdot m^{-1}\partial_\C(\gamma)))-s\cdot \gamma \\ 
	 	&= -(p\circ q)(s\cdot m^{-1}(\partial_\C \gamma))
	 \end{align*} 
	 Using the fact that any term $m^{-1}(\partial_\C\gamma)$ will be labeled by $w_+$, the action of any $s\in \exsltwo_{dg}$ on this term will quotient to 0 under $p\circ q$.  In particular, we have shown that the composition
			$$\RIIa = \C \xrightarrow{q} \C/\C' \xrightarrow{p} \C'' \xrightarrow{f} \RIIe$$
	 is a chain of $L_\infty$-quasi-isomorphisms, since $\RIIe$ has no higher operations.
\end{proof}

\begin{theorem}
	The $L_\infty$-module structure on $\CKh(L)$ is invariant under Reidemeister III.
\end{theorem}
\begin{proof}

	\textbf{Step 1: Overview.} For RIII, the situation is summarized in Figure \ref{RIII}.  We start by decomposing the complexes $\llbracket \raisebox{-.2em}{\resizebox{.15in}{!}{\input{pics/Moves/RIII/RIII-1}}}\rrbracket$ and $\llbracket \raisebox{-.2em}{\resizebox{.15in}{!}{\input{pics/Moves/RIII/RIII-2}}}\rrbracket$ into $\C$ and $\D$ (these are the top left and top right cubes in Figure \ref{RIII}, respectively).  We will then transfer the $L_\infty$-module structures by quasi-isomorphisms $q_2 \circ q_1$ and $p_2 \circ p_1$ to the quotient complexes $(\C/\C')/C''$ and $(\D/\D')/D''$ (the bottom row) and show that these quotients are $L_\infty$-quasi-isomorphic via an $L_\infty$-module map $f$.

\begin{figure}[H]
\begin{tabular}{ccc}
		\begin{tikzcd}[row sep = tiny, column sep=tiny]
			&	|[alias=101]| \RIIILioi  &  & |[alias=111]| \RIIILiii  	  \\ 
				|[alias=100]| \RIIILioo  &  & |[alias=110]| \RIIILiio &   \\ 
			& 	|[alias=001]| \RIIILooi  &  & |[alias=011]| \RIIILoii 	  \\ 
				|[alias=000]| \RIIILooo  &  & |[alias=010]| \RIIILoio &   \\ 
			\arrow[from=000, to=001]
			\arrow[from=000, to=010]
			\arrow[from=000, to=100]
			\arrow[from=001, to=101]
			\arrow[from=001, to=011]
			\arrow[from=010, to=011]
			\arrow[from=010, to=110, crossing over]
			\arrow[from=011, to=111]
			\arrow[from=100, to=101, "\Delta"]
			\arrow[from=100, to=110, crossing over]
			\arrow[from=101, to=111, "m"]
			\arrow[from=110, to=111]
		\end{tikzcd}

		& \qquad \qquad  &

		\begin{tikzcd}[row sep = tiny, column sep=tiny]
			&	|[alias=101]| \RIIIRioi  &  & |[alias=111]| \RIIIRiii  	  \\ 
				|[alias=100]| \RIIIRioo  &  & |[alias=110]| \RIIIRiio &   \\ 
			& 	|[alias=001]| \RIIIRooi  &  & |[alias=011]| \RIIIRoii 	  \\ 
				|[alias=000]| \RIIIRooo  &  & |[alias=010]| \RIIIRoio &   \\ 
			\arrow[from=000, to=001]
			\arrow[from=000, to=010]
			\arrow[from=000, to=100]
			\arrow[from=001, to=101]
			\arrow[from=001, to=011]
			\arrow[from=010, to=011]
			\arrow[from=010, to=110, crossing over]
			\arrow[from=011, to=111]
			\arrow[from=100, to=101]
			\arrow[from=100, to=110, "\Delta", pos=.8, crossing over]
			\arrow[from=101, to=111]
			\arrow[from=110, to=111, "m"]
		\end{tikzcd} \\

		\begin{tikzcd}[row sep = large, column sep=normal]
			|[alias=A]| \  & |[alias=B]| \    \\ 
			|[alias=C]| \  & |[alias=D]| \   
			\arrow[from=A, to=C, "q_1" left, bend right]
			\arrow[from=D, to=B, "i_1" right, bend right]
		\end{tikzcd} & & \begin{tikzcd}[row sep = large, column sep=normal]
			|[alias=A]| \  & |[alias=B]| \    \\ 
			|[alias=C]| \  & |[alias=D]| \   
			\arrow[from=A, to=C, "p_1" left, bend right]
			\arrow[from=D, to=B, "j_1" right, bend right]
		\end{tikzcd}

		\\
		\begin{tikzcd}[row sep = tiny, column sep=tiny]
			&	|[alias=101]| \RIIILioi /_{w_+=0}  &  & |[alias=111]| 0  	  \\ 
				|[alias=100]| \RIIILioo  &  & |[alias=110]| \RIIILiio &   \\ 
			& 	|[alias=001]| \RIIILooi  &  & |[alias=011]| \RIIILoii 	  \\ 
				|[alias=000]| \RIIILooo  &  & |[alias=010]| \RIIILoio &   \\ 
			\arrow[from=000, to=001]
			\arrow[from=000, to=010]
			\arrow[from=000, to=100]
			\arrow[from=001, to=101]
			\arrow[from=001, to=011]
			\arrow[from=010, to=011]
			\arrow[from=010, to=110, crossing over]
			\arrow[from=011, to=111]
			\arrow[from=100, to=101, "\Delta"]
			\arrow[from=100, to=110, crossing over]
			\arrow[from=101, to=111, "m"]
			\arrow[from=110, to=111]
		\end{tikzcd}

		& &

		\begin{tikzcd}[row sep = tiny, column sep=tiny]
			&	|[alias=101]| \RIIIRioi  &  & |[alias=111]| 0  	  \\ 
				|[alias=100]| \RIIIRioo  &  & |[alias=110]| \RIIIRiio /_{w_+=0} &   \\ 
			& 	|[alias=001]| \RIIIRooi  &  & |[alias=011]| \RIIIRoii 	  \\ 
				|[alias=000]| \RIIIRooo  &  & |[alias=010]| \RIIIRoio &   \\ 
			\arrow[from=000, to=001]
			\arrow[from=000, to=010]
			\arrow[from=000, to=100]
			\arrow[from=001, to=101]
			\arrow[from=001, to=011]
			\arrow[from=010, to=011]
			\arrow[from=010, to=110, crossing over]
			\arrow[from=011, to=111]
			\arrow[from=100, to=101]
			\arrow[from=100, to=110, "\Delta", pos=.8, crossing over]
			\arrow[from=101, to=111]
			\arrow[from=110, to=111, "m", start anchor={[xshift=-4ex]}, start anchor={[yshift=-2ex]}]
		\end{tikzcd} \\

		\begin{tikzcd}[row sep = large, column sep=normal]
			|[alias=A]| \  & |[alias=B]| \    \\ 
			|[alias=C]| \  & |[alias=D]| \   
			\arrow[from=A, to=C, "q_2" left, bend right]
			\arrow[from=D, to=B, "i_2" right, bend right]
		\end{tikzcd} & & \begin{tikzcd}[row sep = large, column sep=normal]
			|[alias=A]| \  & |[alias=B]| \    \\ 
			|[alias=C]| \  & |[alias=D]| \   
			\arrow[from=A, to=C, "p_2" left, bend right]
			\arrow[from=D, to=B, "j_2" right, bend right]
		\end{tikzcd}\\

		\begin{tikzcd}[row sep = tiny, column sep=tiny]
			&	|[alias=101]| \RIIILioi/_{w_+=0}  &  & |[alias=111]| 0  	  \\ 
				|[alias=100]| 0  &  & |[alias=110]| \RIIILiio &   \\ 
			& 	|[alias=001]| \RIIILooi  &  & |[alias=011]| \RIIILoii 	  \\ 
				|[alias=000]| \RIIILooo  &  & |[alias=010]| \RIIILoio &   \\ 
			\arrow[from=000, to=001]
			\arrow[from=000, to=010]
			\arrow[from=000, to=100]
			\arrow[from=001, to=101]
			\arrow[from=001, to=011]
			\arrow[from=010, to=011]
			\arrow[from=010, to=110, crossing over]
			\arrow[from=011, to=111]
			\arrow[from=100, to=101, "\Delta"]
			\arrow[from=100, to=110, crossing over]
			\arrow[from=101, to=111, "m"]
			\arrow[from=110, to=111]
			\arrow[from=101, to=110, dashed]
		\end{tikzcd}

		& \begin{tikzcd}[row sep = tiny, column sep=large]
			|[alias=A]| \  & |[alias=B]| \
			\arrow[from=A, to=B, "f"]
		\end{tikzcd} &  

		\begin{tikzcd}[row sep = tiny, column sep=tiny]
			&	|[alias=101]| \RIIIRioi  &  & |[alias=111]| 0  	  \\ 
				|[alias=100]| 0  &  & |[alias=110]| \RIIIRiio/_{w_+=0}  &   \\ 
			& 	|[alias=001]| \RIIIRooi  &  & |[alias=011]| \RIIIRoii 	  \\ 
				|[alias=000]| \RIIIRooo  &  & |[alias=010]| \RIIIRoio &   \\ 
			\arrow[from=000, to=001]
			\arrow[from=000, to=010]
			\arrow[from=000, to=100]
			\arrow[from=001, to=101]
			\arrow[from=001, to=011]
			\arrow[from=010, to=011]
			\arrow[from=010, to=110, crossing over]
			\arrow[from=011, to=111]
			\arrow[from=100, to=101]
			\arrow[from=100, to=110, "\Delta" pos=0.8, crossing over]
			\arrow[from=101, to=111]
			\arrow[from=110, to=111, "m", start anchor={[xshift=-4ex]}, start anchor={[yshift=-2ex]}]
			\arrow[from=110, to=101, dashed, start anchor={[xshift=1ex]}, start anchor={[yshift=-1ex]}]
		\end{tikzcd} \\

\end{tabular}
\caption{The complexes involved in RIII invariance.  We have suppressed the degree shifts.}
\label{RIII}
\end{figure}
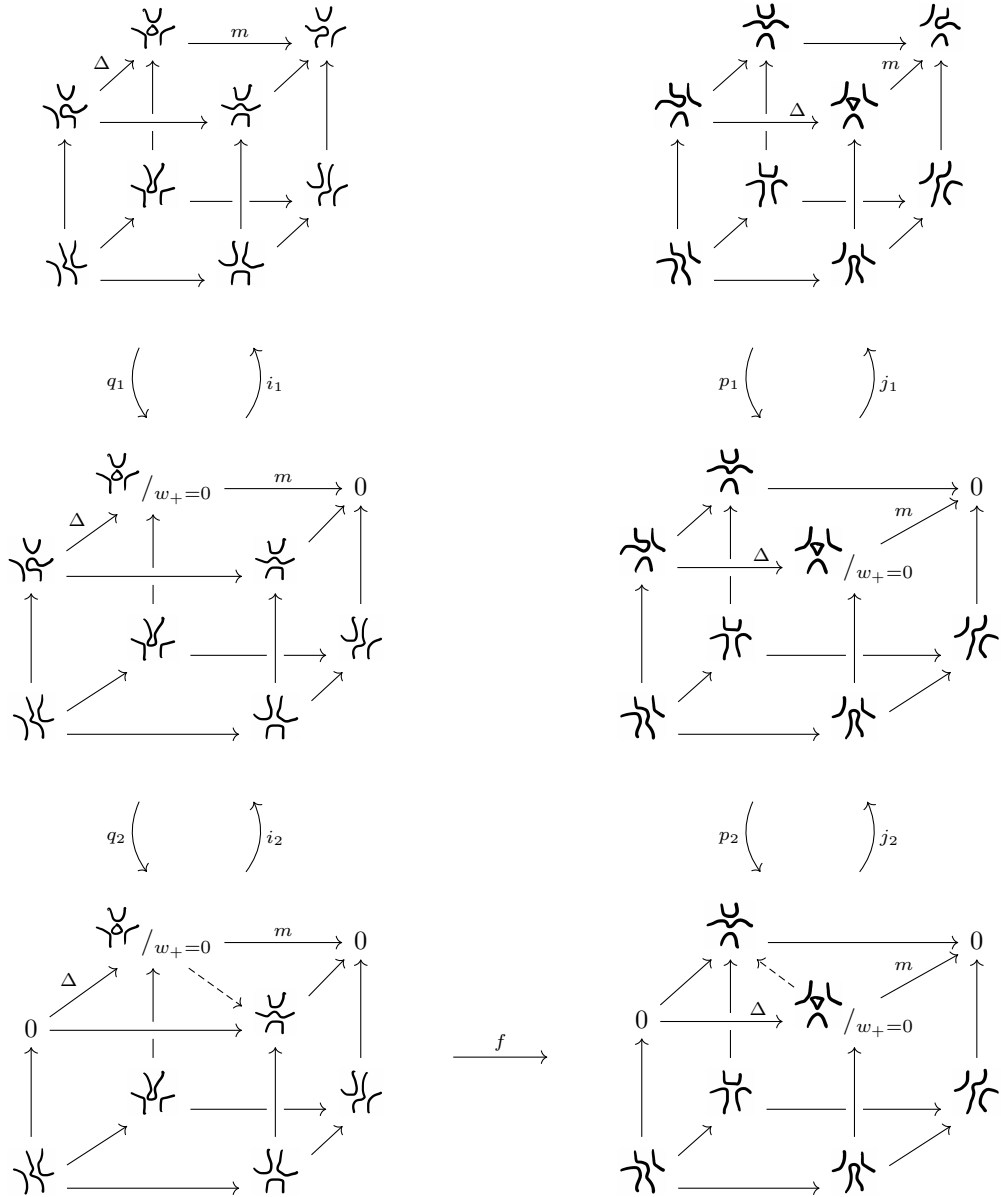

 \textbf{Step 2: The structure on $\C/\C'$ and $\D/\D'$}. Analagous to the RII case, we have subcomplexes $C' \subset C$ and $D' \subset D$, which are $L_\infty$-submodules; see Figure \ref{RIIIsub}.  Because $C'$ and $D'$ are submodules, the quotients $\C/\C'$ and $\D/\D'$ have no higher operations as $L_\infty$-modules.  Alternatively, this quotient structure agrees with the one obtained by using cochain contractions 
 \[
		\begin{tikzcd}[row sep = normal, column sep=normal]
			 \C \arrow[loop left, distance=1em, "H"] \arrow[r, shift left=1, "q_1"] & \C/\C' \arrow[l, shift left=1, "i_1"]
		\end{tikzcd}
	\qquad \mbox{ and } \qquad 
		\begin{tikzcd}[row sep = normal, column sep=normal]
			 \D \arrow[loop left, distance=1em, "K"] \arrow[r, shift left=1, "p_1"] & \D/\D'  \arrow[l, shift left=1, "j_1"]
		\end{tikzcd}
	\]
	to transfer the structure.  Here, the maps $i_1$ and $i_2$ are 
 $$i_1(z) = \begin{cases}
		z, & \mbox{if } z\in 000,001,010,100 \\
		z-m^{-1}(\partial_{\C} z), & \mbox{if } z\in 011,101,110 \\
		0, & \mbox{if } z\in 111  \\
	\end{cases}$$
	and 
$$j_1(z) = \begin{cases}
		z, & \mbox{if } z\in 000,001,010,100 \\
		z-m^{-1}(\partial_{\D} z), & \mbox{if } z\in 011,101,110 \\
		0, & \mbox{if } z\in 111  \\
	\end{cases}$$
	The coordinates above refer to different corners of the cubes, i.e.
	\[
	\begin{tikzcd}[row sep = tiny, column sep=tiny]
			&	|[alias=101]| (101)  &  & |[alias=111]| (111)  	  \\ 
				|[alias=100]| (100)  &  & |[alias=110]| (110) &   \\ 
			& 	|[alias=001]| (001)  &  & |[alias=011]| (011) 	  \\ 
				|[alias=000]| (000)  &  & |[alias=010]| (010) &   \\ 
			\arrow[from=000, to=001]
			\arrow[from=000, to=010]
			\arrow[from=000, to=100]
			\arrow[from=001, to=101]
			\arrow[from=001, to=011]
			\arrow[from=010, to=011]
			\arrow[from=010, to=110, crossing over]
			\arrow[from=011, to=111]
			\arrow[from=100, to=101]
			\arrow[from=100, to=110, crossing over]
			\arrow[from=101, to=111]
			\arrow[from=110, to=111]
	\end{tikzcd}\]

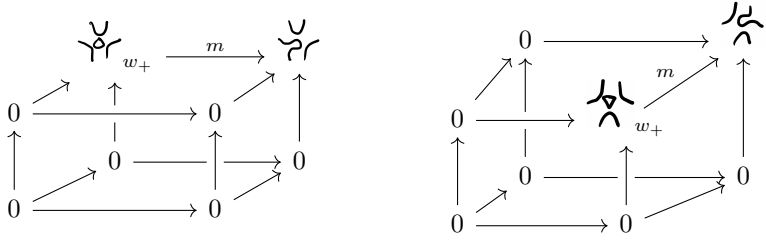
\begin{figure}
\begin{tabular}{ccc}
		\begin{tikzcd}[row sep = tiny, column sep=small]
			&	|[alias=101]| \RIIILioi_{w_+}  &  & |[alias=111]| \RIIILiii  	  \\ 
				|[alias=100]| 0  &  & |[alias=110]| 0 &   \\ 
			& 	|[alias=001]| 0  &  & |[alias=011]| 0 	  \\ 
				|[alias=000]| 0  &  & |[alias=010]| 0 &   \\ 
			\arrow[from=000, to=001]
			\arrow[from=000, to=010]
			\arrow[from=000, to=100]
			\arrow[from=001, to=101]
			\arrow[from=001, to=011]
			\arrow[from=010, to=011]
			\arrow[from=010, to=110, crossing over]
			\arrow[from=011, to=111]
			\arrow[from=100, to=101]
			\arrow[from=100, to=110, crossing over]
			\arrow[from=101, to=111, "m"]
			\arrow[from=110, to=111]
		\end{tikzcd}

		& \qquad \qquad  &

		\begin{tikzcd}[row sep = tiny, column sep=small]
			&	|[alias=101]| 0  &  & |[alias=111]| \RIIIRiii  	  \\ 
				|[alias=100]| 0  &  & |[alias=110]| \RIIIRiio_{w_+} &   \\ 
			& 	|[alias=001]| 0  &  & |[alias=011]| 0 	  \\ 
				|[alias=000]| 0  &  & |[alias=010]| 0 &   \\ 
			\arrow[from=000, to=001]
			\arrow[from=000, to=010]
			\arrow[from=000, to=100]
			\arrow[from=001, to=101]
			\arrow[from=001, to=011]
			\arrow[from=010, to=011]
			\arrow[from=010, to=110, crossing over]
			\arrow[from=011, to=111]
			\arrow[from=100, to=101]
			\arrow[from=100, to=110, crossing over]
			\arrow[from=101, to=111]
			\arrow[from=110, to=111, "m", start anchor={[xshift=-3ex]}, start anchor={[yshift=-2ex]}]
		\end{tikzcd} \\

\end{tabular}
\caption{The complexes $C'$ and $D'$.  The $w_+$ means that the trivial circle is labeled $w_+$.}
\label{RIIIsub}
\end{figure}

\textbf{Step 3: The structure on $\Ciii$ and $\Diii$}.  To go from $\C/\C'$ to $\Ciii$ and $\D/\D'$ to $\Diii$, we identify elements in vertices 101 and 110 by imposing the relation $\beta_1=\tau_1\beta_1$ in $\Ciii$ and the relation $\beta_2=\tau_2\beta_2$ in $\Diii$, analagous to the RII case.  Similar to before, we are not quotienting by a submodule, so we need to transfer the structure from $\C/\C'$ and $\D/\D'$.  To this end, we define maps $i_2$ and $j_2$.  Let
	$$i_2(z) = \begin{cases}
		z, & \mbox{if } z\in 000,010,011,110 \\
		z-\Delta^{-1}(\partial_{\C/\C'} z), & \mbox{if } z\in 001 \\
		0, & \mbox{if } z\in 100, 111 \\
	\end{cases}$$
	Note that if $z \in 101$, then $z$ is equivalent via $\tau_1$ to some element in 110.  Also, let
	$$
	j_2(z) = \begin{cases}
		z, & \mbox{if } z\in 000,001,011,101 \\
		z-\Delta^{-1}(\partial_{\D/\D'} z), & \mbox{if } z\in 010 \\
		0, & \mbox{if } z\in 100, 111 \\
	\end{cases}
	$$
	where we again note that if $z\in 110$, then $z$ is equivalent via $\tau_2$ to some element in 101.  Then, if we define $T: \C/\C' \to \C/\C'$ and $S: \D/\D' \to \D/\D'$ by 
	$$T(z) = \begin{cases}
		\Delta^{-1}(z), & \mbox{if } z\in 101 \\
		0, & \mbox{otherwise }
	\end{cases}\qquad \mbox{ and } \qquad 
	S(z) = \begin{cases}
		\Delta^{-1}(z), & \mbox{if } z\in 110 \\
		0, & \mbox{otherwise } \\
	\end{cases}
	$$
	both  
	\[
		\begin{tikzcd}[row sep = normal, column sep=normal]
			 \C/\C' \arrow[loop left, distance=1em, "T"] \arrow[r, shift left=1, "q_2"] & (\C/\C')/\C''  \arrow[l, shift left=1, "i_2"]
		\end{tikzcd}
	\qquad \mbox{ and } \qquad 
		\begin{tikzcd}[row sep = normal, column sep=normal]
			 \D/\D' \arrow[loop left, distance=1em, "S"] \arrow[r, shift left=1, "p_2"] & (\D/\D')/\D''  \arrow[l, shift left=1, "j_2"]
		\end{tikzcd}
	\]
	satisfy the requirement of a cochain contraction.  In particular, this allows us to transfer the $L_\infty$-module structures from $\C/\C'$ and $\D/\D'$ to their respective quotient complexes.  

	\textbf{Step 4: $\Ciii$ and $\Diii$ have no higher operations}. The next goal is to show that there is no higher structure on $\Ciii$ or $\Diii$.  We will explain the case of $\Ciii$.  The case of $\Diii$ is analagous.  Indeed, because $\C/\C'$ has no higher $L_\infty$-module operations, the transferred structure looks like 
	$$k_n'(x_1,\ldots,x_{n-1},m)= \sum_{\substack{\tau \in S(1, \ldots, 1) \\ i_1= \cdots = i_{n-1} = 1}} q \circ A_{n-1} \circ (\tau^\bullet \otimes i_2)$$  
	See Figure \ref{fig:induced}.  We will show that $q \circ A_{n-1} \circ (\tau^\bullet \otimes i_2) = 0$ for any $\tau \in S_{n-1}$.  That is, for $n\geq 3$, it suffices to show that $q\circ A_{n-1}(x_1,x_2, \ldots, x_{n-1},i_2(m)) = 0$ for any choice of $x_1,x_2,\ldots,x_{n-1} \in \exsltwo_{dg}$, where $i_1, \ldots, i_{n-1}=1$ in the definition of $A_{n-1}$.  

	\textbf{Step 4.1: The case $n>3$}.  We start with the case $n>3$.  Because $T$ is only nonzero on the vertex $101$, for $q\circ A_{n-1}(x_1,x_2, \ldots, x_{n-1},i_2(m))$ to be nonzero, it must contain a nonzero composition
	$$
	\begin{tikzpicture}
	 	\node at (0,0) (n1) {$\RIIILioi/_{w_+=0}$};
	 	\node at (2.5,0) (n2) {$\RIIILioo$};
	 	\node at (5,0) (n3) {$\RIIILioi/_{w_+=0}$};
	 	\node at (7.5,0) (n4) {$\RIIILioo$};
	 	\scriptsize{
  		\node at (0,-1) (blabel) {101};
  		\node at (2.5,-1) (blabel) {100};
  		\node at (5,-1) (blabel) {101};
  		\node at (7.5,-1) (blabel) {100};
	 	\draw [->] (n1) -- (n2) node[midway,above] {$T$};
	 	\draw [->] (n2) -- (n3) node[midway,above] {$x_j$};
	 	\draw [->] (n3) -- (n4) node[midway,above] {$T$};
	 	}
	\end{tikzpicture}
	$$
	Here, the map $x_j$ represents acting by the element $x_j \in \exsltwo_{dg}$.  We will show that if $x_j$ is any element of the basis $\{e,f,h,v_2, v_{-2}, \tilde{v}_0, d,D,x\}$, then this composition is zero.  Indeed, $x_j$ cannot be $e,f,h$, since it must change the homological degree by one to have nonzero image in vertex 101.  Moreover, modulo the relation $w_+=0$, the actions of the elements $v_2, v_{-2}, \tilde{v}_0$, and $D$ are all the zero map.  Finally, if $x_j=-x=[d,D]$, then the component that lies in the vertex 101 is
	$$D_{101}d_{10*} + D_{10*} d_{100} + d_{101}D_{10*} + d_{10*} D_{100}$$
	where, for example, the notation $D_{101}$ represents the component of $D$ that remains in vertex 101, and $d_{10*}$ represents the component of $d$ obtained by acting along the edge $100 \to 101$.  Now we observe that the middle terms $D_{10*} d_{100}$ and $d_{101}D_{10*}$ are both zero, because the relation $w_+=0$ implies that $D_{10*}$ is the zero map.  Also, the terms $D_{101}d_{10*}$ and $d_{10*} D_{100}$ cancel, because $d_{10*}$ just appends a trivial circle labeled $w_-$ to the resolution in vertex 100.

	Therefore, we have reduced the possible nonzero $q\circ A_{n-1}(x_1,x_2, \ldots, x_{n-1},i_2(m))$ to either the case of $q\circ A_{2}(x_1,x_2,i_2(m))$ or $q\circ A_{n-1}(x_1,x_2, \ldots, x_{n-1},i_2(m))$, where $x_2 = \cdots = x_{n-2} = d$.  

	\textbf{Step 4.2: The case $n=3$}. We now examine the case $n=3$.  From the formula for $q\circ A_{2}(x_1,x_2,i_2(m))$, we need $x_1 \cdot i_2(m)$ to be in vertex 101.  This implies that $m$ is either in the vertex $000$ or the vertex $001$.  If $m \in 000$, then the only possibility for $x_1$ is $x_1 = x$.  But then
	\begin{align*}
		x\cdot m &= -[d,D] \cdot m =\partial_0 \partial_0^{Lee} m + \partial_0^{Lee}\partial_0 m 
	\end{align*}
	Since the boundary map $\partial_{*01}$ is a split map, and $w_+=0$ in vertex 101, $\partial_0^{Lee}=0$ along this edge.  So we only have a term $\partial_0\partial_0^{Lee}m$.  Therefore, we need to focus on the composition 
	\[
		\RIIIcase{\RIIILooo}{\RIIILooi}{\RIIILioi}{\RIIILioo}{\RIIILioi}{\RIIILiio}
	\]
	where $t_1, t_2$ can be either merge or split. Let $a,b,c$ denote the circles to which the three strands in vertex 000 belong; see Figure \ref{circles}.  	
	\begin{figure}
		\centering
		\scalebox{.4}{\input{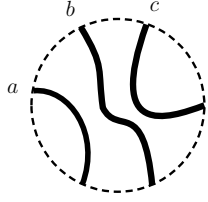}}
		\caption{Each strand in the vertex 000 belongs to a circle.  Denote these circles by $a$, $b$, and $c$.}
		\label{circles}
	\end{figure}
	Then we have four cases: either $a=b=c$, $a=b\neq c$, $a\neq b = c$, or $a\neq b\neq c$.

	\begin{figure}[H]
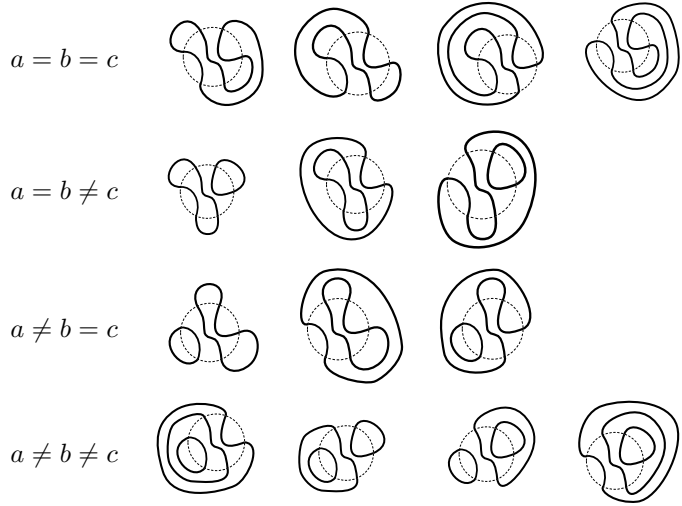

		\centering
			\newcommand{\abcone}{\resizebox{.6in}{!}{\input{pics/circle-cases/abc1}}}
			\newcommand{\abctwo}{\resizebox{.6in}{!}{\input{pics/circle-cases/abc2}}}
			\newcommand{\abcthree}{\resizebox{.6in}{!}{\input{pics/circle-cases/abc3}}}
			\newcommand{\abcfour}{\resizebox{.6in}{!}{\input{pics/circle-cases/abc4}}}

			\newcommand{\abncone}{\resizebox{.6in}{!}{\input{pics/circle-cases/abnc1}}}
			\newcommand{\abnctwo}{\resizebox{.6in}{!}{\input{pics/circle-cases/abnc2}}}
			\newcommand{\abncthree}{\resizebox{.6in}{!}{\input{pics/circle-cases/abnc3}}}

			\newcommand{\anbcone}{\resizebox{.6in}{!}{\input{pics/circle-cases/anbc1}}}
			\newcommand{\anbctwo}{\resizebox{.6in}{!}{\input{pics/circle-cases/anbc2}}}
			\newcommand{\anbcthree}{\resizebox{.6in}{!}{\input{pics/circle-cases/anbc3}}}

			\newcommand{\anbncone}{\resizebox{.6in}{!}{\input{pics/circle-cases/anbnc1}}}
			\newcommand{\anbnctwo}{\resizebox{.6in}{!}{\input{pics/circle-cases/anbnc2}}}
			\newcommand{\anbncthree}{\resizebox{.6in}{!}{\input{pics/circle-cases/anbnc3}}}
			\newcommand{\anbncfour}{\resizebox{.6in}{!}{\input{pics/circle-cases/anbnc4}}}

		\begin{tabular}{m{10ex} m{10ex} m{10ex} m{10ex} m{10ex}}
		 $a=b=c$ &  \resizebox{.6in}{!}{\input{pics/circle-cases/abc1}} & \resizebox{.6in}{!}{\input{pics/circle-cases/abc2}} & \resizebox{.6in}{!}{\input{pics/circle-cases/abc3}} & \resizebox{.6in}{!}{\input{pics/circle-cases/abc4}} \\
		 $a= b\neq c$ &  \resizebox{.6in}{!}{\input{pics/circle-cases/abnc1}} & \resizebox{.6in}{!}{\input{pics/circle-cases/abnc2}} & \resizebox{.6in}{!}{\input{pics/circle-cases/abnc3}} & \\
		 $a\neq b=c$ &  \resizebox{.6in}{!}{\input{pics/circle-cases/anbc1}} & \resizebox{.6in}{!}{\input{pics/circle-cases/anbc2}} & \resizebox{.6in}{!}{\input{pics/circle-cases/anbc3}} &  \\ 
		  $a\neq b\neq c$ &  \resizebox{.6in}{!}{\input{pics/circle-cases/anbnc1}} & \resizebox{.6in}{!}{\input{pics/circle-cases/anbnc2}} & \resizebox{.6in}{!}{\input{pics/circle-cases/anbnc3}} & \resizebox{.6in}{!}{\input{pics/circle-cases/anbnc4}} 
		\end{tabular}

		\caption{This picture shows all possible configurations of the circles $a$, $b$, and $c$.}
		\label{circlecases}
	\end{figure}

	We have not drawn the basepoint, which can be anywhere outside of the dashed circles. We have also not drawn the possible other circles coming from the other crossing resolutions.  

	\textbf{Step 4.2.1: $m\in 000$ and $a=b=c$}. If $a=b=c$, then in each case, $t_1$ is a split map and $t_2$ is a merge map.  Because $\partial_0^{Lee}$ needs to be nonzero, we must label our circle by $w_-$.  This forces a labeling of $w_+\otimes w_+$ in 100; see Figure \ref{RIIIcase1}.

	\newcommand{\abci}{\includegraphics[width=0.35in]{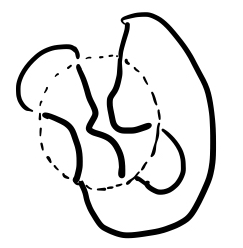}}
	\newcommand{\abcii}{\includegraphics[width=0.35in]{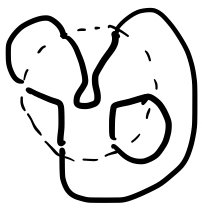}}
	\newcommand{\abciii}{\includegraphics[width=0.35in]{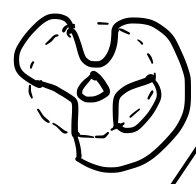}}
	\newcommand{\abciv}{\includegraphics[width=0.35in]{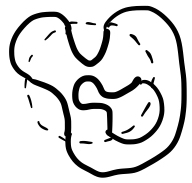}}
	\newcommand{\abcv}{\includegraphics[width=0.35in]{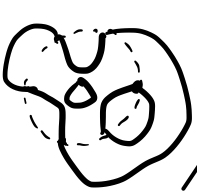}}
	\newcommand{\abcvi}{\includegraphics[width=0.35in]{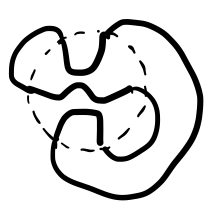}}

	\newcommand{\xabci}{\resizebox{.35in}{!}{\input{pics/RIII-L/abc/000}}}
	\newcommand{\xabcii}{\resizebox{.35in}{!}{\input{pics/RIII-L/abc/001}}}
	\newcommand{\xabciii}{\resizebox{.35in}{!}{\input{pics/RIII-L/abc/101}}}
	\newcommand{\xabciv}{\resizebox{.35in}{!}{\input{pics/RIII-L/abc/100}}}
	\newcommand{\xabcv}{\resizebox{.35in}{!}{\input{pics/RIII-L/abc/101}}}
	\newcommand{\xabcvi}{\resizebox{.35in}{!}{\input{pics/RIII-L/abc/110}}}

	\begin{figure}[H]
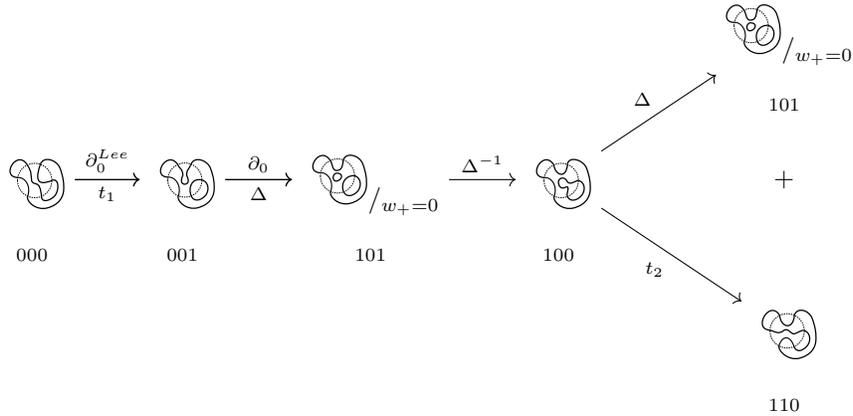

		\centering
		\[
		\RIIIcase{\resizebox{.35in}{!}{\input{pics/RIII-L/abc/000}}}{\resizebox{.35in}{!}{\input{pics/RIII-L/abc/001}}}{\resizebox{.35in}{!}{\input{pics/RIII-L/abc/101}}}{\resizebox{.35in}{!}{\input{pics/RIII-L/abc/100}}}{\resizebox{.35in}{!}{\input{pics/RIII-L/abc/101}}}{\resizebox{.35in}{!}{\input{pics/RIII-L/abc/110}}}
		\]
		\caption{The first of four cases with $a=b=c$.  In each case, the labeling of the circle in 000 must be $w_-$, which forces a labeling of $w_+ \otimes w_+$ in 100.}
		\label{RIIIcase1}
	\end{figure}
	The possibilities for $x_2$ are $\{e,f,h,v_2, v_{-2}, \tilde{v}_0, d,D,x\}$.  It cannot be $e,f,h$, since $x_2$ must change the homological be degree by one.  Moreover, $v_2, v_{-2}$ and $\tilde{v}_0$ are each the 0 map, since we are only involving trivial circles.  The labeling $w_+\otimes w_+$ implies that $D$ is the 0 map.  Finally, the terms obtained from acting by either $d$ or $x$ will cancel when we quotient to $\Ciii$.  For example, if we act by $d$, then the relation $\beta_1=\tau_1\beta_1$ identifies the terms obtained by acting by $d_{10*}$ and $d_{1*0}$, and so they will cancel.  On the other hand, if we act by $x$, the terms we obtain in vertices 101 and 110 are
	$$
		\underbrace{d_{10*} D_{100} + d_{101} D_{10*} + D_{10*}d_{100} + D_{101}d_{10*}}_{\mbox{vertex } 101} + \underbrace{d_{1*0} D_{100} + d_{110} D_{1*0} + D_{1*0}d_{100} + D_{110}d_{1*0}}_{\mbox{vertex } 110}
	$$
	Now, the terms involving $D_{10*}$ and $D_{1*0}$ are zero, because $w_+ = 0$ in vertex 101 and both circles in vertex 100 are labeled by $w_+$.  We are left with
	\begin{align*}
		 d_{10*} D_{100} + D_{101}d_{10*} + d_{1*0} D_{100} + D_{110}d_{1*0}\\
	\end{align*}
	Because of the $w_+$ labelings in vertex 100, the only nonzero parts of $D_{101}$ and $D_{110}$ come from applying $D$ amongst the other circles in the resolution.  It follows that $D_{101}d_{10*}$ and $D_{110}d_{1*0}$ will be identified when we quotient, and so they will cancel.  The $d_{10*} D_{100}$ and $d_{1*0} D_{100}$ terms will also cancel.

	\textbf{Step 4.2.2: $m\in 000$ and $a=b\neq c$}.  If $a=b\neq c$, then in each case, $t_1$ is a split map and $t_2$ is a split map.  Again, we need to involve trivial circles for $t_1$, otherwise $\partial_0^{Lee}=0$; see Figure \ref{RIIIcase2}.

	\newcommand{\abnci}{\includegraphics[width=0.35in]{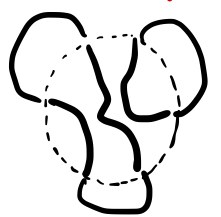}}
	\newcommand{\abncii}{\includegraphics[width=0.35in]{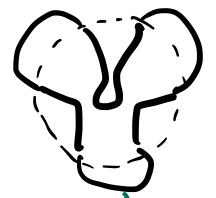}}
	\newcommand{\abnciii}{\includegraphics[width=0.35in]{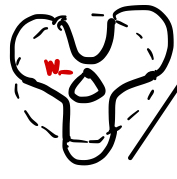}}
	\newcommand{\abnciv}{\includegraphics[width=0.35in]{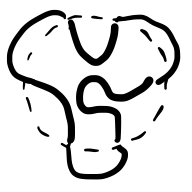}}
	\newcommand{\abncv}{\includegraphics[width=0.35in]{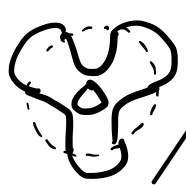}}
	\newcommand{\abncvi}{\includegraphics[width=0.35in]{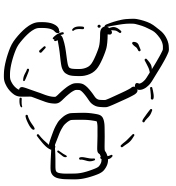}}

	\newcommand{\xabnci}{\resizebox{!}{.35in}{\input{pics/RIII-L/abnc/000}}}
	\newcommand{\xabncii}{\resizebox{!}{.35in}{\input{pics/RIII-L/abnc/001}}}
	\newcommand{\xabnciii}{\resizebox{!}{.35in}{\input{pics/RIII-L/abnc/101}}}
	\newcommand{\xabnciv}{\resizebox{!}{.35in}{\input{pics/RIII-L/abnc/100}}}
	\newcommand{\xabncv}{\resizebox{!}{.35in}{\input{pics/RIII-L/abnc/101}}}
	\newcommand{\xabncvi}{\resizebox{!}{.35in}{\input{pics/RIII-L/abnc/110}}}

	\begin{figure}[H]
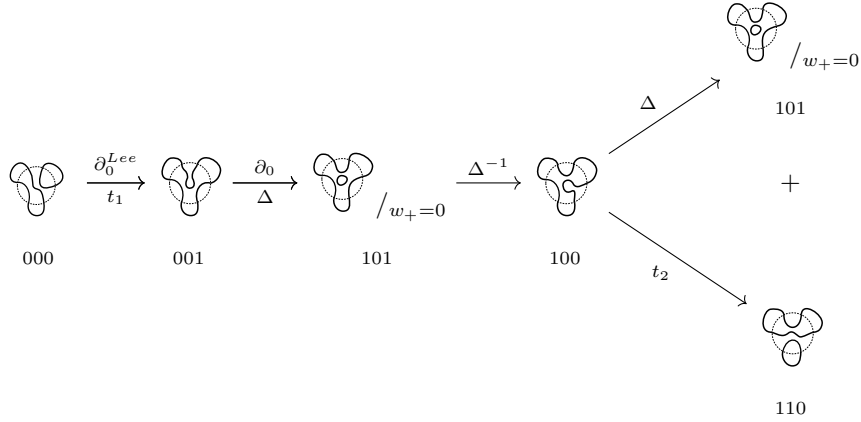

	$$
		\RIIIcase{\resizebox{!}{.35in}{\input{pics/RIII-L/abnc/000}}}{\resizebox{!}{.35in}{\input{pics/RIII-L/abnc/001}}}{\resizebox{!}{.35in}{\input{pics/RIII-L/abnc/101}}}{\resizebox{!}{.35in}{\input{pics/RIII-L/abnc/100}}}{\resizebox{!}{.35in}{\input{pics/RIII-L/abnc/101}}}{\resizebox{!}{.35in}{\input{pics/RIII-L/abnc/110}}}
	$$
	\caption{The first of three cases with $a=b\neq c$.  The labeling of the circle in 000 must be $w_-\otimes w_-$, which forces a labeling of $w_+$ in 100.}
	\label{RIIIcase2}
	\end{figure}	

	Since the circle in vertex 100 must be labeled by $w_+$, by a similar argument to the case of $a=b=c$, acting by $v_2, v_{-2}, \tilde{v}_0, D$ are all $0$, and the terms obtained from acting by either $d$ or $x$ will cancel when we quotient to $\Ciii$.  Seeing that the terms will cancel in the quotient if we act by $x$ in vertex 100 is slightly different than before.  To see this explicitly, we start as in the case of $a=b=c$ by examining the terms 
	$$
		\underbrace{d_{10*} D_{100} + d_{101} D_{10*} + D_{10*}d_{100} + D_{101}d_{10*}}_{\mbox{vertex } 101} + \underbrace{d_{1*0} D_{100} + d_{110} D_{1*0} + D_{1*0}d_{100} + D_{110}d_{1*0}}_{\mbox{vertex } 110}
	$$
	Now, $d_{10*} D_{100}$ and $d_{1*0} D_{100}$ will cancel in the quotient.  Also, $D_{10*}$ is the zero map due to the relation $w_{+} = 0$ in vertex 101.  The $w_+$ label implies that $d_{110}D_{1*0}$ is zero.  It remains to show that the terms
	\begin{align*}
		D_{101}d_{10*} + D_{1*0}d_{100} + D_{110}d_{1*0}
	\end{align*}
	cancel.  Label the circles in 100 by $c_1 \otimes \cdots \otimes c_n \otimes w_+$.  The idea is to show that part of $D_{110}d_{1*0}$ will cancel with $D_{101}d_{10*}$ (the part involving the $c_i$ themselves) and that the rest will cancel with $D_{1*0}d_{100}$ (the part involving the $w_+$).  Indeed, we may write the $D_{110}d_{1*0}$ term as
	\begin{align*}
		D_{110}d_{1*0} &= D_{110}(c_1 \otimes \cdots \otimes c_n \otimes w_+ \otimes w_-) + D_{110}(c_1 \otimes \cdots \otimes c_n \otimes w_- \otimes w_+) \\ 
		& = D^c_{110}(c_1 \otimes \cdots \otimes c_n \otimes w_+ \otimes w_-) + D^c_{110}(c_1 \otimes \cdots \otimes c_n \otimes w_- \otimes w_+) \\ 
		& \quad + D^{w_+}_{110}(c_1 \otimes \cdots \otimes c_n \otimes w_+ \otimes w_-) + D^{w_+}_{110}(c_1 \otimes \cdots \otimes c_n \otimes w_- \otimes w_+) \\ 
		& \quad + D^{w_-}_{110}(c_1 \otimes \cdots \otimes c_n \otimes w_+ \otimes w_-) + D^{w_-}_{110}(c_1 \otimes \cdots \otimes c_n \otimes w_- \otimes w_+) \\ 
	\end{align*}
	where $D^c_{110}$ is the part of $D_{110}$ that involves only crossings among the circles $c_1, \ldots, c_n$, $D^{w_+}$ is the part of $D_{110}$ that involves only crossings with the circle labeled $w_+$, and $D^{w_-}$ is the part of $D_{110}$ that involves only crossings with the circle labeled $w_-$.  By the definition of the Lee differential, the labels imply $D^{w_+}_{110}(c_1 \otimes \cdots \otimes c_n \otimes w_+ \otimes w_-) = D^{w_+}_{110}(c_1 \otimes \cdots \otimes c_n \otimes w_- \otimes w_+) = 0$, and so
	\begin{align*}
		D_{110}d_{1*0} &=  D^c_{110}(c_1 \otimes \cdots \otimes c_n \otimes w_+ \otimes w_-) + D^c_{110}(c_1 \otimes \cdots \otimes c_n \otimes w_- \otimes w_+) \\ 
		& \quad + D^{w_-}_{110}(c_1 \otimes \cdots \otimes c_n \otimes w_- \otimes w_+) + D^{w_-}_{110}(c_1 \otimes \cdots \otimes c_n \otimes w_+ \otimes w_-) 
	\end{align*}
	On the other hand, $D_{101}d_{10*}$ can be written as 
	\begin{align*}
		D_{101}(c_1\otimes \cdots \otimes c_n \otimes w_+ \otimes w_-) &= D_{101}^c(c_1\otimes \cdots \otimes c_n \otimes w_+ \otimes w_-) \\ 
		& \quad + D^w_{101}(c_1\otimes \cdots \otimes c_n \otimes w_+ \otimes w_-)
	\end{align*}
	where $D^w_{101}$ is the part of $D_{101}$ involving a crossing with either the (outermost) circle labeled $w_+$ or the circle labeled $w_-$.  The $w_+$ label together with the relation $w_+=0$ in vertex 101 implies that $D^w_{101}(c_1\otimes \cdots \otimes c_n \otimes w_+ \otimes w_-) = 0$.  In the quotient $\Ciii$, $D_{101}^c(c_1\otimes \cdots \otimes c_n \otimes w_+ \otimes w_-)$ is identified with 
	$$D_{110}(c_1 \otimes \cdots \otimes c_n \otimes w_+ \otimes w_-) + D_{110}(c_1 \otimes \cdots \otimes c_n \otimes w_- \otimes w_+)$$
	Therefore, it remains to examine the $D_{1*0}d_{100}$ term, which we may write as 
	\begin{align*}
		D_{1*0}d_{100}(c_1\otimes \cdots \otimes c_n \otimes w_+) &= D_{1*0}d^c_{100}(c_1\otimes \cdots \otimes c_n \otimes w_+) \\ 
		&\quad + D_{1*0}d^w_{100}(c_1\otimes \cdots \otimes c_n \otimes w_+)
	\end{align*}
	Because of the $w_+$ label, $D_{1*0}d^c_{100}(c_1\otimes \cdots \otimes c_n \otimes w_+) = 0$, and so it remains to show that 
	$$D_{1*0}d^w_{100}(c_1\otimes \cdots c_n \otimes w_+)$$
	and
	$$D^{w_+}_{110}(c_1 \otimes \cdots \otimes c_n \otimes w_- \otimes w_+) + D^{w_-}_{110}(c_1 \otimes \cdots \otimes c_n \otimes w_+ \otimes w_-)$$  
	cancel in $\Ciii$.  This is indeed the case since to compute $D^{w_+}_{110}(c_1 \otimes \cdots \otimes c_n \otimes w_- \otimes w_+)$, we need only consider crossings where either a circle $c_i$ labeled $w_-$ merges with the $w_-$ or the circle labeled $w_-$ splits.  The same is true to compute $D^{w_-}_{110}(c_1 \otimes \cdots \otimes c_n \otimes w_+ \otimes w_-)$.  On the other hand, to compute $D_{1*0}d^w_{100}(c_1\otimes \cdots \otimes c_n \otimes w_+)$ we again have two cases.  The first case consists of crossings where a circle $c_i$ labeled $w_-$ merges with the $w_+$.  These terms will cancel with those from the first case above.  The second case consists of the crossings where a $w_+$ splits to $w_- \otimes w_+ + w_+ \otimes w_-$.  These terms will cancel with the second case above.

    \textbf{Step 4.2.3: $m\in 000$ and $a\neq b=c$}.  We can now study the case $a\neq b = c$.  In this scenario, $t_1$ is a split map and $t_2$ is a merge map; see Figure \ref{RIIIcase3}.

	\newcommand{\anbci}{\includegraphics[width=0.35in]{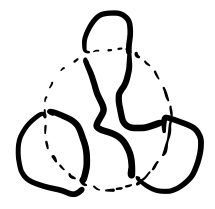}}
	\newcommand{\anbcii}{\includegraphics[width=0.35in]{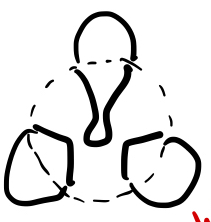}}
	\newcommand{\anbciii}{\includegraphics[width=0.35in]{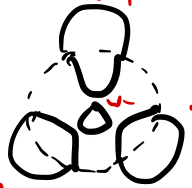}}
	\newcommand{\anbciv}{\includegraphics[width=0.35in]{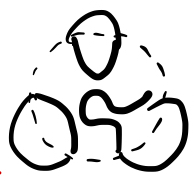}}
	\newcommand{\anbcv}{\includegraphics[width=0.35in]{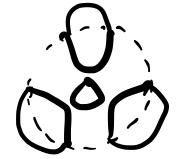}}
	\newcommand{\anbcvi}{\includegraphics[width=0.35in]{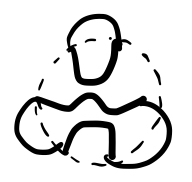}}

	\newcommand{\xanbci}{\resizebox{!}{.35in}{\input{pics/RIII-L/anbc/000}}}
	\newcommand{\xanbcii}{\resizebox{!}{.35in}{\input{pics/RIII-L/anbc/001}}}
	\newcommand{\xanbciii}{\resizebox{!}{.35in}{\input{pics/RIII-L/anbc/101}}}
	\newcommand{\xanbciv}{\resizebox{!}{.35in}{\input{pics/RIII-L/anbc/100}}}
	\newcommand{\xanbcv}{\resizebox{!}{.35in}{\input{pics/RIII-L/anbc/101}}}
	\newcommand{\xanbcvi}{\resizebox{!}{.35in}{\input{pics/RIII-L/anbc/110}}}

	\begin{figure}[H]
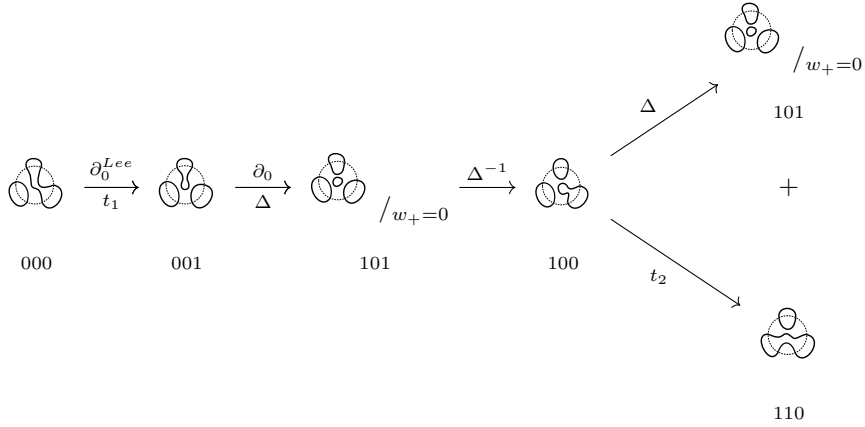

		\[
		\RIIIcase{\resizebox{!}{.35in}{\input{pics/RIII-L/anbc/000}}}{\resizebox{!}{.35in}{\input{pics/RIII-L/anbc/001}}}{\resizebox{!}{.35in}{\input{pics/RIII-L/anbc/101}}}{\resizebox{!}{.35in}{\input{pics/RIII-L/anbc/100}}}{\resizebox{!}{.35in}{\input{pics/RIII-L/anbc/101}}}{\resizebox{!}{.35in}{\input{pics/RIII-L/anbc/110}}}
		\]
		\caption{The first of three cases with $a\neq b= c$.  The labeling of the circle in 000 must be $w_-\otimes w_-$, which forces a labeling of $w_- \otimes w_+ \otimes w_+$ in 100.}
		\label{RIIIcase3}
	\end{figure}
	In each case, the labeling in $000$ must be $w_- \otimes w_-$, and this forces a labeling of $w_- \otimes w_+ \otimes w_+$ in 100 in each case.  Again, $v_2, v_{-2}, \tilde{v}_0, D$ are all $0$, and a similar argument shows that the terms obtained from acting by either $d$ or $x$ will cancel when we quotient to $\Ciii$.  

	\textbf{Step 4.2.4: $m\in 000$ and $a\neq b\neq c$}.  Finally, if $a\neq b \neq c$, then $t_1$ is a merge map and $t_2$ is a merge map; see Figure \ref{RIIIcase4}.

	\newcommand{\anbnci}{\includegraphics[width=0.45in]{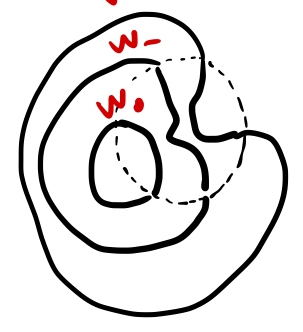}}
	\newcommand{\anbncii}{\includegraphics[width=0.45in]{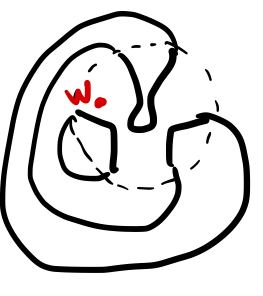}}
	\newcommand{\anbnciii}{\includegraphics[width=0.45in]{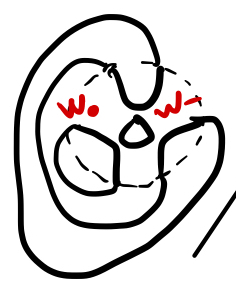}}
	\newcommand{\anbnciv}{\includegraphics[width=0.45in]{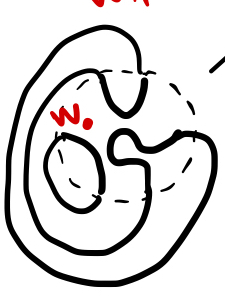}}
	\newcommand{\anbncv}{\includegraphics[width=0.45in]{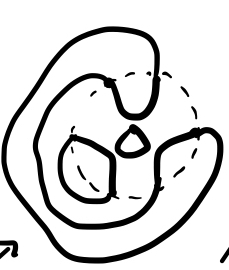}}
	\newcommand{\anbncvi}{\includegraphics[width=0.45in]{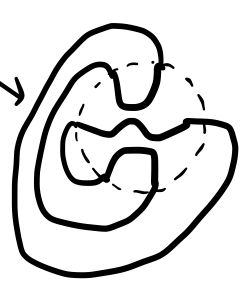}}

	\newcommand{\xanbnci}{\resizebox{!}{.35in}{\input{pics/RIII-L/anbnc/000}}}
	\newcommand{\xanbncii}{\resizebox{!}{.35in}{\input{pics/RIII-L/anbnc/001}}}
	\newcommand{\xanbnciii}{\resizebox{!}{.35in}{\input{pics/RIII-L/anbnc/101}}}
	\newcommand{\xanbnciv}{\resizebox{!}{.35in}{\input{pics/RIII-L/anbnc/100}}}
	\newcommand{\xanbncv}{\resizebox{!}{.35in}{\input{pics/RIII-L/anbnc/101}}}
	\newcommand{\xanbncvi}{\resizebox{!}{.35in}{\input{pics/RIII-L/anbnc/110}}}

	\begin{figure}[H]
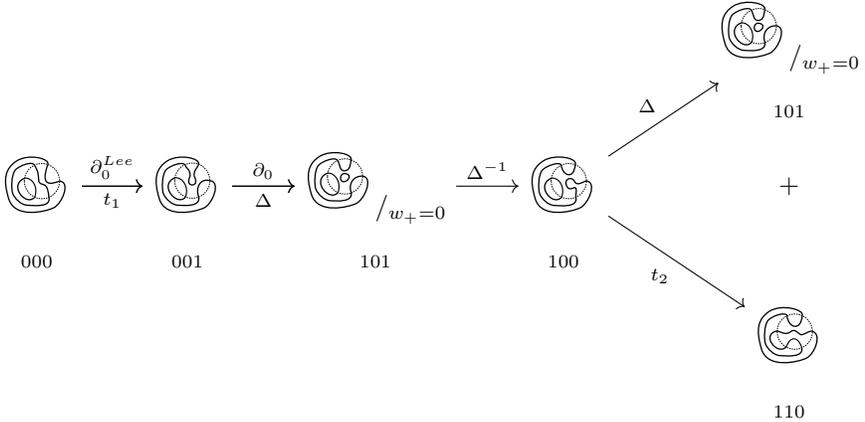

		\[
		\RIIIcase{\resizebox{!}{.35in}{\input{pics/RIII-L/anbnc/000}}}{\resizebox{!}{.35in}{\input{pics/RIII-L/anbnc/001}}}{\resizebox{!}{.35in}{\input{pics/RIII-L/anbnc/101}}}{\resizebox{!}{.35in}{\input{pics/RIII-L/anbnc/100}}}{\resizebox{!}{.35in}{\input{pics/RIII-L/anbnc/101}}}{\resizebox{!}{.35in}{\input{pics/RIII-L/anbnc/110}}}
		\]
		\caption{The first of four cases of $a\neq b \neq c$.  The labeling of the circle in 000 must be $w_-\otimes w_- \otimes w_\bullet$, which forces a labeling of $w_+ \otimes w_\bullet$ in 100.}
		\label{RIIIcase4}
	\end{figure}

	In each case, the labeling in $000$ must be $w_-\otimes w_- \otimes w_\bullet$, where $w_\bullet$ denotes that the innermost circle can be labeled either $w_+$ or $w_-$.  This forces a labeling of $w_+ \otimes w_\bullet$ in 100 in all cases.  For the last time, we verify that $v_2, v_{-2}, \tilde{v}_0, D$ are all $0$, and a similar argument shows that the terms obtained from acting by either $d$ or $x$ will cancel when we quotient to $\Ciii$.  To summarize, we have thus shown that $q\circ A_{2}(x_1,x_2,i_2(m))=0$ for all $m\in 000$ and $x_1,x_2 \in \exsltwo_{dg}$, and we conclude that $k_3'(x_1,x_2,m) = 0$ for all $m\in 000$ and $x_1,x_2 \in \exsltwo_{dg}$ as well.

	\textbf{Step 4.2.5: $m\in 001$}.  We next examine $q\circ A_{2}(x_1,x_2,i_2(m))=0$ in the case $m \in 001$.  The relevant composition in the RIII cube is given in Figure \ref{m001}.

	\begin{figure}[H]
		\[
		\begin{tikzcd}[row sep = tiny, column sep=normal]
			& \RIIILioo \arrow[rd, "\Delta", swap] \arrow[rd, "x_1"] & & & \RIIILioi/_{w_+=0} \\
			\RIIILooi \arrow[ru] \arrow[rd] &  & \RIIILioi/_{w_+=0} \arrow[r, "T"]  & \RIIILioo \arrow[ru, "x_2"] \arrow[rd, "x_2"] &  \ \\ 
			& \RIIILooi \arrow[ru, "\Delta", swap] \arrow[ru, "x_1"] & & & \RIIILiio 
		\end{tikzcd}
		\]
		\caption{The relevant part of the RIII cube.  If we start with an element in 001, $i_2: \Ciii \to \C/\C'$ gives a sum of elements in 001 and 100.  We then act by $x_1$, apply the homotopy $T$, act by $x_2$, and then quotient back to $\Ciii$.}
		\label{m001}
	\end{figure}
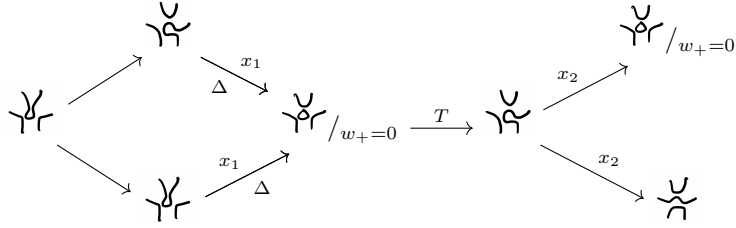

	As before, the possibilities for $x_1$ are $\{e,f,h,v_2, v_{-2}, \tilde{v}_0, d,D,x\}$.  Because $x_1$ needs to increase the homological degree of $m$, it cannot be $e,f$, or $h$.  Since we are working modulo $w_+=0$, both $\partial^{Lee}$ and $\partial_-$ are the zero map, and so $D$, $v_2$, $v_{-2}$, and $\tilde{v}_0$ are all the zero map.  Moreover, $x_1$ cannot be $d$, since the resolutions in 001 and 100 have the same label, which means that they will cancel when mapped to 101.  Similarly, the fact that both resolutions have the same label also implies that the terms in $x=-[d,D]$ will cancel.  We conclude that $q\circ A_{2}(x_1,x_2,i_2(m))=0$ for all $m\in 001$ and $x_1,x_2 \in \exsltwo_{dg}$, and so we have thus shown that $k_3'(x_1,x_2,m) = 0$ on $\Ciii$.

	\textbf{Step 4.2.6: Conclusion}. From the above case analysis, the only possible higher operation is $k_n'$ for $n>3$, which could include a nonzero term of $q\circ A_{n-1}(x_1,x_2, \ldots, x_{n-1},i_2(m))$ with $x_2 = \cdots = x_{n-2} = d$.  But because $d$ is just the inverse to the chain homotopy $T$, this will cycle the module element back and forth between vertices 101 and 100.  In particular, $q\circ A_{n-1}(x_1, x_2, \ldots, x_{n-1},i_2(m)) = q\circ A_2(x_1,x_{n-1},i_2(m))$, which we have already shown is zero.  We conclude that $k'_n=0$ on $\Ciii$ for $n>3$, and this completes the proof that $\Ciii$ has no higher operations.  The symmetry between $\Ciii$ and $\Diii$ implies that $\Diii$ also has no higher operations.

	\textbf{Step 5: The cubes $\Ciii$ and $\Diii$ are quasi-isomorphic}. It remains to construct the map $f: \Ciii \to \Diii$ and show that it respects the (trivial) $L_\infty$-module structures.  Indeed, in $(\C/\C')/\C''$, each $\beta_1 \in 101$ is equivalent via $\tau_1$ to some $\gamma_1\in 110$.  The map $f$ will send an element in $110$ to itself, but as an element of $101$ in $(\D/\D')/\D''$, and it will keep the bottom layer of the cube fixed.  This is an isomorphism on spaces, and Bar-Natan checks that this map is a chain map; see \cite{BN02}.  So, for $s\in \exsltwo_{dg}$ and $x\in \Ciii$ we need to compare $f(s\cdot x)$ and $s\cdot f(x)$, where the module structure is $s\cdot x = q_2(s\cdot i_2(x)) = q_2(q_1(s\cdot i_1(i_2(x)))$.  

	\textbf{Step 5.1: The case $s\in \{e,f,h\}$}.  Suppose that $s\in \{e,f,h\}$.  First we examine the case where $z$ is on the bottom face of the cube.  If $z$ is in $000$ or $010$, then
	\begin{align*}
		q_2(q_1(s\cdot i_1(i_2(z))) = q_2(q_1(s\cdot z))
	\end{align*}
	Note that we abuse notation and think of $z$ as an element of $\C$ on the right-hand side.  If $z$ is in $001$, then 
	\begin{align*}
		q_2(q_1(s\cdot i_1(i_2(z))) &= q_2(q_1(s\cdot i_1(z-\Delta^{-1}(\partial_{\C/\C'}z)))) \\ 
		&= q_2(q_1(s\cdot z-s\cdot \Delta^{-1}(\partial_{\C/\C'}z))) \\ 
		&= q_2(q_1(s\cdot z)) - q_2q_1(s\cdot \Delta^{-1}(\partial_{\C/\C'}z)) \\ 
		&= q_2(q_1(s\cdot z))
	\end{align*}
	because $s\cdot \Delta^{-1}(\partial_{\C/\C'}z)$ is in $100$, which quotients to 0. If $z$ is in $011$, then 
	\begin{align*}
		q_2(q_1(s\cdot i_1(i_2(z))) &= q_2(q_1(s\cdot i_1(z))) \\ 
		&= q_2(q_1(s\cdot(z-m^{-1}\partial_\C z))) \\ 
		&= q_2(q_1(s\cdot z)) - q_2q_1(s\cdot m^{-1}(\partial_{\C}z)) \\ 
		&= q_2(q_1(s\cdot z))
	\end{align*}
	because $s\cdot m^{-1}(\partial_{\C}z)$ is labeled $w_+$, which quotients to 0.  A similar argument shows that $s\cdot z=p_2(p_1(s\cdot z))$, if $z$ is thought of as an element of $\Diii$.  Since $f$ is the identity on the bottom face, it follows that $s\cdot f(z) = f(s\cdot z)$ for $z\in 000,010,001,100$.

	If $z$ is on the top face, we need only consider the case $z \in 110$, since any element in $101$ is equivalent to some $z \in 110$.  Then in $(\C/\C')/\C''$,
	\begin{align*}
		s\cdot z = q_2(q_1(s\cdot i_1(i_2(z))) &= q_2(q_1(s\cdot i_1(z))) \\ 
		&= q_2(q_1(s\cdot(z-m^{-1}\partial_\C z))) \\ 
		&= q_2(q_1(s\cdot z)) - q_2q_1(s\cdot m^{-1}(\partial_{\C}z)) \\ 
		&= q_2(q_1(s\cdot z))
	\end{align*}
	because $s\cdot m^{-1}(\partial_{\C}z)$ is labeled $w_+$, which quotients to 0.  On the other hand, if we consider $z$ as an element of $101$ in $(\D/\D')/\D''$,
	\begin{align*}
		s\cdot z = p_2(p_1(s \cdot j_1(j_2( z))) &= p_2(p_1(s\cdot j_1(z))) \\ 
		&= p_2(p_1(s\cdot(z-m^{-1}\partial_{\D} z))) \\ 
		&= p_2(p_1(s\cdot z)) - p_2p_1(s\cdot m^{-1}(\partial_{\D}z)) \\ 
		&= p_2(p_1(s\cdot z))
	\end{align*}
	Since $f$ identically maps elements in $110$ in $\Ciii$ to those in $101$ in $\Diii$, it follows that $s\cdot f(z) = f(s \cdot z)$ on the top face.

	\textbf{Step 5.2: The case $s \in \{v_2,v_{-2}, \tilde{v}_0, d, D\}$}.  Suppose that $s \in \{v_2,v_{-2}, \tilde{v}_0, d, D\}$.  We again start with the case that $z$ is on the bottom face of $\Ciii$. 	 The cases $z \in 000$ and $z\in 011$ are straightforward to check, since $f$ is the identity on the bottom face.  If $z\in 001$, then in $\Ciii$, 
	 \begin{align*}
	 	q_2(q_1(s\cdot i_1(i_2(z))) &= q_2(q_1(s\cdot i_1(z-\Delta^{-1}(\partial_{\C/\C'}z)))) \\ 
		&= q_2(q_1(s\cdot z-s\cdot \Delta^{-1}(\partial_{\C/\C'}z))) \\ 
		&= q_2(q_1(s\cdot z)) - q_2(q_1(s\cdot \Delta^{-1}(\partial_{\C/\C'}z))) \\ 
		&= q_2(q_1(s_0 \cdot z + s_{*01}\cdot z + s_{0*1}\cdot z))-q_2(q_1(s_{10*}\cdot \Delta^{-1}\partial_{\C/\C'}z + s_{1*0}\cdot \Delta^{-1}\partial_{\C/\C'}z))
	  \end{align*} 
	  and in $\Diii$,
	  \begin{align*}
	  	s\cdot f(z) &= p_2(p_1(s_0\cdot f(z) + s_{*01}\cdot f(z) + s_{0*1}\cdot f(z)))
	  \end{align*}
	  and we must show that $f$ maps the former to the latter.
	  Indeed, in $\Ciii$, the terms $q_2q_1(s_{*01} \cdot z)$ and $q_2q_1(s_{10*}\cdot \Delta^{-1}(\partial_{\C/\C'} z))$ will cancel.  This is because $\Delta^{-1}(\partial_{\C/\C'} z)$ has the same labeling as $z$, and both maps to $101$ are split maps.  Furthermore, $q_2q_1(s_{1*0} \cdot \Delta^{-1}(\partial_{\C/\C'} z))$ in $\Ciii'$ will be mapped via $f$ to $s_{*01}\cdot f(z)$.  This is because $\Delta^{-1}(\partial_{\C/\C'} z)$ has the same labeling as $z$ and the maps $\partial_{1*0}$ in $\Ciii$ and $\partial_{*01}$ in $\Diii$ are of the same type (i.e. they are either both merge or both split), meaning $s$ will act the same across these maps.  The case $z \in 010$ is analogous.  Next, suppose that $z$ is in the top face of the cube.  If $z \in 110$, then
	\begin{align*}
		s\cdot z = q_2(q_1(s\cdot i_1(i_2(z))) &= q_2(q_1(s\cdot i_1(z))) \\ 
		&= q_2(q_1(s\cdot(z-m^{-1}\partial_\C z))) \\ 
		&= q_2(q_1(s_{110} \cdot z)) - q_2q_1(s_{101} \cdot m^{-1}(\partial_{\C}z)) \\ 
		&= q_2(q_1(s_{110} \cdot z))
	\end{align*}
	 where $s_{110}\cdot z$ is the part of $s\cdot z$ that remains in $110$ and $s_{101} \cdot m^{-1}(\partial_{\C}z)$ is the part of $s \cdot m^{-1}(\partial_{\C}z)$ that remains in $101$.  But the latter quotients to 0, as it is labeled by $w_+$.  On the other hand, if we consider $z$ as an element of $101$ in $(\D/\D')/\D''$,
	\begin{align*}
		s\cdot z = p_2(p_1(s \cdot j_1(j_2( z))) &= p_2(p_1(s\cdot j_1(z))) \\ 
		&= p_2(p_1(s\cdot(z-m^{-1}\partial_{\D} z))) \\ 
		&= p_2(p_1(s_{101}\cdot z)) - p_2p_1(s_{110}\cdot m^{-1}(\partial_{\D}z)) \\ 
		&= p_2(p_1(s_{101}\cdot z))
	\end{align*}
	where $s_{101}\cdot z$ is the part of $s\cdot z$ that remains in $101$ and $s_{110} \cdot m^{-1}(\partial_{\D}z)$ is the part of $s \cdot m^{-1}(\partial_{\D}z)$ that remains in $110$. Similar to before, the latter quotients to 0, as it is labeled by $w_+$.  Since $f$ identically maps elements in $110$ in $\Ciii$ to those in $101$ in $\Diii$, we conclude that $s\cdot f(z) = f(s \cdot z)$ on the top face.

	\textbf{Step 5.3: The case $s = x = -[d,D]$}.  Finally, suppose that $s = x = -[d,D]$.  For $z\in \Ciii$,
		$$f(s\cdot z) = f((-dD-Dd)\cdot z) = -df(D\cdot z)-Df(d\cdot z) = (-dD-Dd)\cdot f(z) = s\cdot f(z)$$

	\textbf{Step 5.4: Conclusion}. To summarize, we have shown that for every element $s$ in a basis of $\exsltwo_{dg}$, $f(s\cdot z) = s\cdot f(z)$.  We conclude that $f: \Ciii \to \Diii$ is an $L_\infty$-module quasi-isomorphism, and so up to quasi-isomorphism, the $L_\infty$-module structure on $\CKh(L)$ is invariant under the Reidemeister III move. 	
\end{proof}

\subsection{Invariance of the $\exsltwo$-module structure}

Now that we have shown the invariance of the $\exsltwo_{dg}$ $L_\infty$-module structure on $\CKh(L)$ under Reidemeister moves, we can show that the $\exsltwo$ $L_\infty$-module structure on $\CKh(L)$ is invariant as well.

\begin{theorem}
\label{thm:wedgeinvariance}
	Up to $L_\infty$-quasi-isomorphism, the $\exsltwo$ $L_\infty$-module structure is invariant under Reidemeister moves.
\end{theorem}
\begin{proof}
	This follows from the fact that the $\exsltwo$ $L_\infty$-module structure on $\CKh(L)$ was obtained from the $\exsltwo_{dg}$ $L_\infty$-module structure by restricting scalars through an $L_\infty$-algebra homomorphism $I: H(\exsltwo_{dg}) \to \exsltwo_{dg}$.  In particular, restricton of scalars preserves $L_\infty$-quasi-isomorphisms (see \cite{D22}), so applying the restriction of scalars functor to the quasi-isomorphisms constructed in the proof of invariance for $\exsltwo_{dg}$ yields quasi-isomorphisms of these complexes considered as $L_\infty$-modules over $H(\exsltwo)$.  Finally, the $\exsltwo$ $L_\infty$-module structure is invariant, since $\exsltwo$ is an $L_\infty$-subalgebra of $H(\exsltwo_{dg})$.
\end{proof}

\section{The $L_\infty$-module structure on $\AKh(L)$}

In this section, we explain how the annular Khovanov homology $\AKh(L)$ has an $L_\infty$-module structure that is invariant under Reidemeister moves.

\begin{theorem}
	Let $L$ be an annular link.  There is an $L_\infty$-module structure on $\AKh(L)$, invariant under Reidemeister moves.  It is well-defined up to $L_\infty$-quasi-isomorphism.  
\end{theorem}
\begin{proof}
The situation can be summarized by the following diagram.
\[
	\begin{tikzcd}[row sep = small, column sep=small]
		\exsltwo_{dg} \arrow[d] \arrow[r] & \CKh(L) \arrow[d]  \\ 
		\exsltwo \arrow[r, dashed] \arrow[ru, dashed] & \AKh(L)  \\ 
	\end{tikzcd}
\]


Theorem \ref{thm:wedgeinvariance} proved that, up to $L_\infty$-quasi-isomorphism, the $L_\infty$-module structure on $\CKh(L)$ over $\exsltwo$ is invariant under Reidemeister moves.  By Theorem \ref{thm:transfer}, $\AKh(L)$ inherits an $L_\infty$-module structure over $\exsltwo$ via any choice of chain contraction $\CKh(L)\to \AKh(L)$.  By Lemma \ref{lem:transfermorphism}, $\AKh(L)$ is quasi-isomorphic to $\CKh(L)$, so if $L$ and $L'$ differ by Reidemeister moves, we have the following diagram:
\[
	\begin{tikzcd}[row sep = small, column sep=small]
		\CKh(L) \arrow[d, "\cong"] \arrow[r, "\cong"] & \CKh(L') \arrow[d, "\cong"]  \\ 
		\AKh(L) \arrow[r, dashed] & \AKh(L')  \\ 
	\end{tikzcd}
\]
This shows that $\AKh(L)$ and $\AKh(L')$ are quasi-isomorphic as $L_\infty$-modules over $\exsltwo$, and so this $L_\infty$-module structure is well-defined up to $L_\infty$-quasi-isomorphism. 
\end{proof}

\section{Examples}
 
In this section, we explore the $L_\infty$-module structure of several knots and links. 

\begin{example}
	Let $L$ be any link in $S^3$ where $\partial^{Lee}$ is nonzero on Khovanov homology.  We may view $L$ as an annular link by placing the basepoint away from the link.  If we denote the $L_\infty$-module operation on $\AKh(L)$ by $k_n$, $\partial_0^{Lee}$ will yield a corresponding nontrivial $k_3(v_2,v_{-2},m)$ on $\AKh(L)$.  Indeed, the $L_\infty$-module structure on $\AKh(L)$ is induced from a cochain contraction $\begin{tikzcd}[row sep = normal, column sep=normal]
			 \CKh(L) \arrow[loop left, distance=1em, "T"] \arrow[r, shift left=1, "q"] & \AKh(L).  \arrow[l, shift left=1, "i"]
		\end{tikzcd}$
	If $k'_n$ is the $L_\infty$-module operation on $\CKh(L)$, the following equation gives a formula for $k_3(x_1,x_2,m)$.
	$$k_3(x_1,x_2,m) = k'_3(x_1,x_2,m) + q\circ k'_2(x_1,T\circ k_2(x_2,i(m))) + q\circ k'_2(x_2, T\circ k_2(x_1,i(m)))$$
	The $k'_2$ operations vanish because all of the circles involved are trivial.
\end{example}

\begin{example}
	In the above example, suppose we put an unknot $U$ around the basepoint.  Let $w \in \AKh(L)$ be a generator on which $\partial_0^{Lee}$ acts nontrivially.  After choosing a cochain contraction that respects $\CKh(U\sqcup L) = V \otimes \CKh(L)$, then in $\AKh(U \sqcup L)$, the generators $v_\pm \otimes w$ have both nontrivial $k_2$ and $k_3$ actions.
\end{example}

\begin{example}
	The left-handed trefoil with the basepoint in the center is an example of a knot $K$ where $\AKh(K)$ has both nontrivial $k_2$ and $k_3$ operations; see Figure \ref{fig:trefoil}.

	\newcommand{\Tooo}{\resizebox{.75in}{!}{\input{pics/trefoil-cube/000}}}
	\newcommand{\Tooi}{\resizebox{.75in}{!}{\input{pics/trefoil-cube/001}}}
	\newcommand{\Toio}{\resizebox{.75in}{!}{\input{pics/trefoil-cube/010}}}
	\newcommand{\Tioo}{\resizebox{.75in}{!}{\input{pics/trefoil-cube/100}}}
	\newcommand{\Toii}{\resizebox{.75in}{!}{\input{pics/trefoil-cube/011}}}
	\newcommand{\Tioi}{\resizebox{.75in}{!}{\input{pics/trefoil-cube/101}}}
	\newcommand{\Tiio}{\resizebox{.75in}{!}{\input{pics/trefoil-cube/110}}}
	\newcommand{\Tiii}{\resizebox{.75in}{!}{\input{pics/trefoil-cube/111}}}

	\begin{figure}[h]
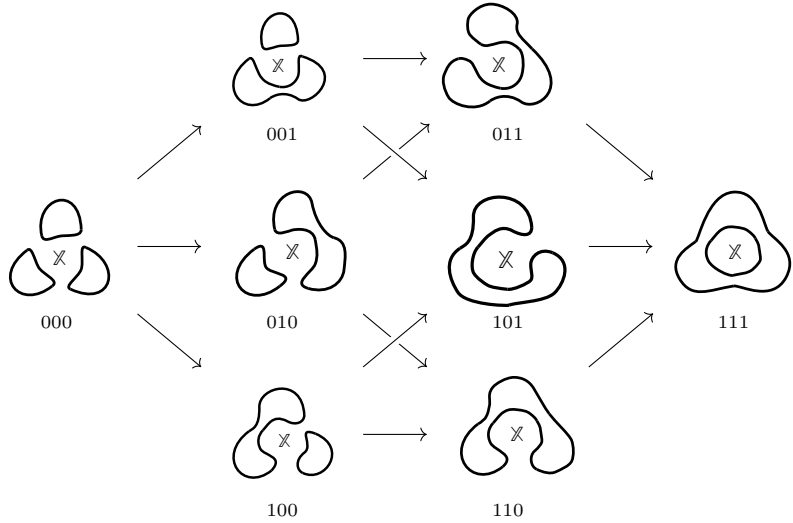

		\centering 
	\begin{tikzpicture}[on top/.style={preaction={draw=white,-,line width=#1}},on top/.default=4pt]
	 	\node at (0,0) (n000) {\resizebox{.75in}{!}{\input{pics/trefoil-cube/000}}};
	 	\node at (3,2.5) (n001) {\resizebox{.75in}{!}{\input{pics/trefoil-cube/001}}};
	 	\node at (3,0) (n010) {\resizebox{.75in}{!}{\input{pics/trefoil-cube/010}}};
	 	\node at (3,-2.5) (n100) {\resizebox{.75in}{!}{\input{pics/trefoil-cube/100}}};
	 	\node at (6,2.5) (n011) {\resizebox{.75in}{!}{\input{pics/trefoil-cube/011}}};
	 	\node at (6,0) (n101) {\resizebox{.75in}{!}{\input{pics/trefoil-cube/101}}};
	 	\node at (6,-2.5) (n110) {\resizebox{.75in}{!}{\input{pics/trefoil-cube/110}}};
	 	\node at (9,0) (n111) {\resizebox{.75in}{!}{\input{pics/trefoil-cube/111}}};
	 	\scriptsize{
  		\node at (0,-1) (blabel) {000};
  		\node at (3,1.5) (blabel) {001};
  		\node at (3,-1) (blabel) {010};
  		\node at (3,-3.5) (blabel) {100};
  		\node at (6,1.5) (blabel) {011};
  		\node at (6,-1) (blabel) {101};
  		\node at (6,-3.5) (blabel) {110};
  		\node at (9,-1) (blabel) {111};
	 	\draw [->] (n000) -- (n001);
	 	\draw [->] (n000) -- (n010);
	 	\draw [->] (n000) -- (n100);
	 	
	 	\draw [->] (n001) -- (n011);
	 	\draw [->] (n100) -- (n110);
	 	\draw [->] (n010) -- (n011);
	 	\draw [->] (n010) -- (n110);
	 	\draw [->, on top] (n001) -- (n101);
	 	\draw [->, on top] (n100) -- (n101);

	 	\draw [->] (n011) -- (n111);
	 	\draw [->] (n101) -- (n111);
	 	\draw [->] (n110) -- (n111);
	 	}
	\end{tikzpicture}
		\caption{The cube of resolutions for the left-handed trefoil knot with basepoint in the center.}
		\label{fig:trefoil}
	\end{figure}
	Indeed, in resolution 000, $k_3(v_2, v_{-2}, w_-\otimes w_- \otimes w_-)=\partial_0^{Lee}(w_-\otimes w_- \otimes w_-)$ is nonzero in homology.  Also, the usual module action of $\exsltwo$ acts nontrivially on the generator $v_+\otimes v_+$ in resolution 111.  Notice that the mirror (the right-handed trefoil) does not have a nontrivial $k_3$ operation in the lowest homological degree. 
\end{example}

\begin{example}
	The above example generalizes to any torus knot or link where the basepoint is in the center.  If every boundary map coming from the lowest homological degree is a merge map, the resolution with each circle labeled $w_-$ will have a nontrivial $k_3$ operation, and the module will act nontrivially on a generator in the highest homological degree.
\end{example}

The examples above illustrate that for an annular link $L$, $\AKh(L)$ can have both nontrivial $k_2$ and $k_3$ operations.  In the case where $L$ is a split link (i.e., at least one component is disjoint), it is further possible for a specific generator to have both nontrivial $k_2$ and $k_3$ operations.  On the other hand, we end this section with a question regarding non-split links.

\begin{question}
	Does there exist a non-split link $L \subset A \times I$ such that $\AKh(L)$ contains a homology class on which the $k_2$ and $k_3$ operations of $\exsltwo$ are nontrivial?  In other words, for a non-split annular link $L$, can there exist $m \in \AKh(L)$ and $x,y_1,y_2 \in \exsltwo$ such that $k_2(x,m)\neq 0$ and $k_3(y_1,y_2,m)\neq 0$?
\end{question}


\section{Appendix}

This appendix contains graphical representations of the formulas presented in the proof of Theorem \ref{thm:transfer}.

	\begin{figure}[H]
		\centering
		\scalebox{1}{\input{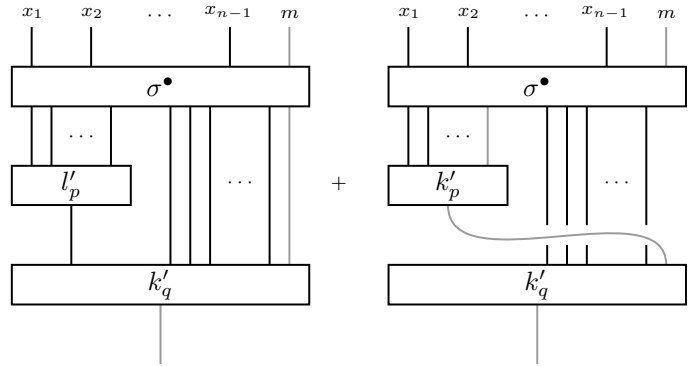}}

		\caption{A graphical depiction of the $L_\infty$-module relation, as in \cite{D22}.}
		\label{fig:modrel}
	\end{figure}

\vfill

	\begin{figure}[H]
		\centering
		\scalebox{1}{\input{pics/transfer/xStep-1.tex}}
		\caption*{Step 1. We start with the terms on the left-hand side of the $L_\infty$-module relation and replace $k'_q$ with its definition.}
	\end{figure}

\vfill

\newpage

	\begin{figure}[H]
		\centering
		\scalebox{1}{\input{pics/transfer/xStep-2.tex}}
		\caption*{Step 2. By the definition of unshuffle, the $l_p$ term in Step 1 goes to the first element in one of the boxes of size $i_1, \ldots, i_t$ determined by $\tau$.  This observation allows us to combine $\sigma$ and $\tau$ into an unshuffle $\eta$.}

	\end{figure}

	\begin{figure}[H]
		\centering
		\scalebox{.9}{\input{pics/transfer/xStep-3a.tex}} \hfill
		\scalebox{.9}{\input{pics/transfer/xStep-3b.tex}}

		\caption*{Step 3. After unpacking the definition of $A_t$, the left-hand side in the above figure represents the second term in the proof.  The cases where $l=1$, $l=t$, and $p=n-1$ are not pictured here.  We obtain the right-hand side after applying the $L_\infty$-module relation.}
	\end{figure}

	\begin{figure}[H]
		\centering
		\scalebox{.6}{\input{pics/transfer/Step-4/1.tex}}\hfill
		\scalebox{.6}{\input{pics/transfer/Step-4/2.tex}}\hfill
		\scalebox{.6}{\input{pics/transfer/Step-4/3.tex}}\\ \vspace{.2in}
		\scalebox{.6}{\input{pics/transfer/Step-4/4.tex}}\hfill
		\scalebox{.6}{\input{pics/transfer/Step-4/5.tex}}\hfill
		\scalebox{.6}{\input{pics/transfer/Step-4/6.tex}}\\ \vspace{.2in}
		\scalebox{.6}{\input{pics/transfer/Step-4/7.tex}}\hfill
		\scalebox{.6}{\input{pics/transfer/Step-4/8.tex}}\hfill
		\scalebox{.6}{\input{pics/transfer/Step-4/9.tex}}\\ \vspace{.2in}
		\scalebox{.6}{\input{pics/transfer/Step-4/10.tex}}\hfill
		\scalebox{.6}{\input{pics/transfer/Step-4/11.tex}}\hfill
		\scalebox{.6}{\input{pics/transfer/Step-4/12.tex}}
		\caption*{Step 4.  Combine the permuations $\psi$ and $\eta$ into $\kappa$.  There are four terms in step three, and each row in this figure represents one of those terms, where the cases $p=1$, $2\leq p \leq s$, and $p=s+1$ are considered separately (pictured left to right).  For $1<p<s+1$, we may combine the $k_p$ and $k_{s-p+2}$ operations into the $A_t$ operation to obtain the formulas in Step 4.}
	\end{figure}

	\begin{figure}[H]
		\centering
		\scalebox{1}{\input{pics/transfer/xStep-5.tex}}
		\caption*{Step 5.  We can combine some of the terms in Step 4.  In the graphic for Step 4 above, label the terms in the first row by 1, 2, 3, the terms in the second row by 4, 5, 6, the terms in the third row by 7, 8, 9, and the terms in the last row by 10, 11, 12.  Then terms 4 and 7 combine to give the first term above on the left.  The middle term is obtained by combining terms 3 and 6.  The last term is obtained by combining terms 2, 5, 8, and 11.  Moreover, the terms 1 and 10 combine, and so too do 9 and 12, but these two cases are not pictured here.}
	\end{figure}

	\begin{figure}[H]
		\centering
		\scalebox{.8}{\input{pics/transfer/xStep-7a.tex}}\hfill
		\scalebox{.95}{\input{pics/transfer/xStep-7b.tex}}\hfill
		\scalebox{.95}{\input{pics/transfer/xStep-7c.tex}}
		\caption*{Step 7. Focusing now on the right-hand side of the original $L_\infty$-module relation, we substitute for $k'_n$ using its definition.  On the left is the case $2\leq p \leq n-1$, in the center is the case $p=1$, and on the right is the case $p=n-1$.  After using the fact that $i\circ k'_1 = k_1 \circ i$ and $k'_1 \circ q = q\circ k_1$, we obtain the formulas in Step 7.}
	\end{figure}

	\begin{figure}[H]
		\centering
		\scalebox{1}{\input{pics/transfer/Step-4.tex}}

		\caption*{Step 8.  Combine $\sigma$, $\alpha$, and $\beta$ into one unshuffle $\theta$.  Drawn above is the case $2\leq p \leq n-1$.  The cases of $p=1$ and $p=n-1$ are omitted.}
	\end{figure}

\begin{figure}[H]
		\centering
		\scalebox{1}{\input{pics/transfer/Step-8.tex}}
		\caption*{Step 9.  In Step 8, we can replace $i\circ q$ with $\Id_M + k_1 \circ T + T \circ k_1$.  The result is precisely what we had in Step 5.  Again, the cases of $p=1$ and $p=n-1$ are not included in this picture. }
\end{figure}

\bibliographystyle{plain}
\bibliography{annular}

\end{document}

%% file: pics/Moves/RIII/RIII-1.tex
\tikzset{every picture/.style={line width=0.75pt}} 

\begin{tikzpicture}[x=0.75pt,y=0.75pt,yscale=-1,xscale=1]

\draw [line width=6]    (55,140) .. controls (77.14,154.93) and (81.64,157.43) .. (109.28,168.84) ;
\draw [line width=6]    (214,166) .. controls (239.64,158.93) and (244.28,156.84) .. (267.28,140.84) ;
\draw [line width=6]    (141.64,177.79) .. controls (165.79,178.58) and (161.64,177.79) .. (179.64,177.29) ;
\draw [line width=6]    (117,34) .. controls (117.28,77.84) and (125.28,86.84) .. (144.28,104.84) ;
\draw [line width=6]    (203.28,32.84) .. controls (202.28,115.48) and (117.28,113.48) .. (118.28,246.84) ;
\draw [line width=6]    (177,136) .. controls (210.28,156.84) and (203.28,228.48) .. (203.28,248.48) ;

\end{tikzpicture}

%% file: pics/Moves/RIII/RIII-2.tex
\tikzset{every picture/.style={line width=0.75pt}} 

\begin{tikzpicture}[x=0.75pt,y=0.75pt,yscale=-1,xscale=1]

\draw [color={rgb, 255:red, 0; green, 0; blue, 0 }  ,draw opacity=1 ][line width=6]    (55,140) .. controls (75.28,124.02) and (86.28,119.02) .. (110.28,112.02) ;
\draw [color={rgb, 255:red, 0; green, 0; blue, 0 }  ,draw opacity=1 ][line width=6]    (213.28,114.02) .. controls (237.28,121.02) and (245.28,127.02) .. (267.28,140.84) ;
\draw [color={rgb, 255:red, 0; green, 0; blue, 0 }  ,draw opacity=1 ][line width=6]    (141.64,102.79) .. controls (152.64,102.11) and (160.64,102.11) .. (175.28,102.02) ;
\draw [color={rgb, 255:red, 0; green, 0; blue, 0 }  ,draw opacity=1 ][line width=6]    (117,34) .. controls (116.28,92.02) and (121.28,120.02) .. (144.28,145.02) ;
\draw [color={rgb, 255:red, 0; green, 0; blue, 0 }  ,draw opacity=1 ][line width=6]    (203.28,32.84) .. controls (202.28,170.02) and (117.28,163.02) .. (118.28,246.84) ;
\draw [color={rgb, 255:red, 0; green, 0; blue, 0 }  ,draw opacity=1 ][line width=6]    (177.28,176.02) .. controls (205.28,202.02) and (203.28,228.48) .. (203.28,248.48) ;

\end{tikzpicture}